\documentclass[a4paper, 11pt]{article}
\usepackage[utf8]{inputenc}
\usepackage{geometry}
\usepackage{graphicx}
\usepackage{authblk} 
\usepackage{amsmath}
\usepackage{amssymb}
\usepackage{amsthm}

\usepackage{comment}

\usepackage{mathrsfs} 
\usepackage{stmaryrd} 
\usepackage{multirow,makecell}
\usepackage{slashed}
\usepackage{xcolor}
\usepackage{mathtools}
\usepackage{xstring}
\usepackage{csquotes} 
\usepackage{hyperref}
\hypersetup{colorlinks,citecolor=blue,filecolor=blacklinkcolor=red, urlcolor=magenta,linkcolor=red}

\usepackage[normalem]{ulem} 

\usepackage{tikz,tikz-cd} 
\tikzset{line/.pic={
\draw[#1] (-0.3,0) -- (0.3,0);},
dot/.pic={
\fill[#1] (0,0) circle(.6mm);}
 }
 \usepackage{xfp}
\usepackage{subcaption}
\usepackage{placeins}
\newtheorem{defi}{Definition}[section]
\newtheorem{lemm}{Lemma}[section]

\newtheorem{theo}{Theorem}[section]
\newtheorem*{theo*}{Theorem}
\newtheorem{prop}{Proposition}[section]
\theoremstyle{remark}
\newtheorem{rema}{Remark}[section]

\numberwithin{equation}{section}

\usepackage[giveninits=true,hyperref,backend=biber,articlein=false,style=ext-alphabetic,maxalphanames=10,maxcitenames=10,maxbibnames=10]{biblatex}
\AtEveryBibitem{%
  \clearfield{issn} 
  \clearfield{Urldate}
  \clearfield{urlmonth}
    \clearfield{month}
  \clearfield{urlyear}
  \clearfield{PrimaryClass}
  \clearfield{eprint}
  \ifentrytype{online}{}{
    \clearfield{url}
  }
}

\DeclareMathOperator{\sgn}{sgn}
\newcommand{\dd}{\textnormal{d}}
\newcommand{\N}{\mathbb{N}}
\newcommand{\Z}{\mathbb{Z}}
\newcommand{\R}{\mathbb{R}}
\newcommand{\C}{\mathbb{C}}
\newcommand{\Weyl}[4]{W_{#1#2\phantom{#3}#4}^{\phantom{#1#2}#3}}

\DeclareMathOperator{\detd}{\textnormal{\textbf{det}}}

\newcommand{\Q}{\M} 
\newcommand{\Qint}{\Mint} 
\newcommand{\E}{\mathcal{E}}
\newcommand{\T}{\mathcal{T}}
\newcommand{\Tscr}{\mathscr{T}}
\newcommand\rotpi{\rotatebox[origin=c]{180}{$\Pi$}}
\newcommand{\M}{M} 
\newcommand{\Mint}{\mathring{M}}
\newcommand{\B}{{\partial M}} 

\newcommand{\eB}{{\widetilde{\mathcal{H}}_{\pm}}} 

\newcommand{\D}{\mathbf{D}}
\newcommand{\ei}[1]{{\uppercase \mathcal{#1}}} 
\newcommand{\pti}[1]{{\uppercase {\mathbf{#1}}}} 
\newcommand{\st}{\,\,} 
\newcommand*{\dscale}[1][2]{%
\IfEqCase{#1}{%
{-1}{|\lambda_0|^{-\frac{1}{2}}}%
{1}{|\lambda_0|^\frac{1}{2}}%
{2}{\lambda_0}}}

\newcommand{\pscale}{\hat{\sigma}} 
\newcommand{\pnabla}{\hat{\nabla}}  
\newcommand{\bnabla}{\nabla} 
\newcommand{\pg}{\hat{g}} 
\newcommand{\bP}{P} 
\newcommand{\bdh}{\bar{h}} 

\newcommand{\n}{{\bar{n}}}
\newcommand{\pgctform}{\hat{\upsilon}} 
\newcommand{\mJ}{\hat{J}} 
\newcommand{\mf}{\hat{f}} 
\newcommand{\mvphi}{\hat{\varphi}} 
\newcommand{\mL}{\hat{L}} 

\newcommand{\Ti}{T_{\mathrm{i}}}
\newcommand{\Spi}{S_{\mathrm{pi}}}
\newcommand{\scri}{\mathscr{I}}

\title{Projective and Carrollian geometry at time/space-like infinity on projectively compact Ricci flat Einstein manifolds.}

\author[1]{Jack Borthwick\thanks{\href{mailto:jack.borthwick@mcgill.ca}{jack.borthwick@mcgill.ca}}}
\author[2]{Yannick Herfray\thanks{\href{mailto:yannick.herfray@univ-tours.fr}{yannick.herfray@univ-tours.fr}}}
\affil[1]{Department of Mathematics and Statistics, McGill University}
\affil[2]{Institut Denis Poisson, Université de Tours}

\date{\today}

\begin{document}
\maketitle
{\let\thefootnote\relax}
\begin{abstract}
In this article we discuss how to construct canonical \emph{strong} Carrollian geometries at time/space like infinity of projectively compact Ricci flat Einstein manifolds $(M,g)$ and discuss the links between the underlying projective structure of the metric $g$. The obtained Carrollian geometries are determined by the data of the projective compactification.  The key idea to achieve this is to consider a new type of Cartan geometry based on a non-effective homogeneous model for projective geometry. We prove that this structure is a general feature of projectively compact Ricci flat Einstein manifolds. 
\end{abstract}
\tableofcontents
\section{Introduction}
A geometric structure on a manifold with boundary $M$ can induce distinct geometries on the interior $\Mint$ and on the boundary $\B$. An important example of this phenomenon is when $M$ is equipped with a Cartan geometry admitting a holonomy reduction. As proved by \v{C}ap--Gover--Hammerl in \cite{cap_holonomy_2014} this indeed leads to a curved orbit decomposition of $\M$, that one can take to be compatible with the canonical decomposition $\M=\Mint\cup \B$. This provides a framework in which one can explore, from a geometric perspective, different compactification procedures which appear in modern day physics. These are of interest since such compactifications typically underlie holographic duality principles.

Holonomy reductions of projective structures~\cite{Cap:2014aa,Cap:2014ab, cap_projective_2016-1,Flood:2018aa}  provide particularly rich orbit decompositions in which there is an interesting interplay between conformal, projective and pseudo-riemannian geometry. This is nicely illustrated by the diagram in Figure~\ref{fig:CProjMinkowski} which illustrates the \emph{projective} compactification of Minkowski spacetime $(\mathbb{M}^{d+1}, \eta=-\dd t^2+\textrm{Euc}_{\R^d})$, as obtained from a curved orbit decomposition of the projective sphere (which generalises to an appropriate curved setting~\cite{Flood:2018aa}). \v{C}ap and Gover have shown (see~\cite[Proposition 2.4]{Cap:2014ab}) that boundary points of this compactification are, in a precise sense, endpoints of the geodesics of $\mathbb{M}^{d+1}$. The decomposition of the boundary into three orbits then corresponds to the three possible types of geodesic curves reaching it, which are distinguished by the \enquote{norm} $\eta(u,u)$ of their tangent vector: timelike $\eta(u,u)<0$, null $\eta(u,u)=0$ and spacelike $\eta(u,u)>0$ curves.  Accordingly, the three boundary orbits correspond to ``timelike'', ``null'' and ``spacelike" infinity respectively.
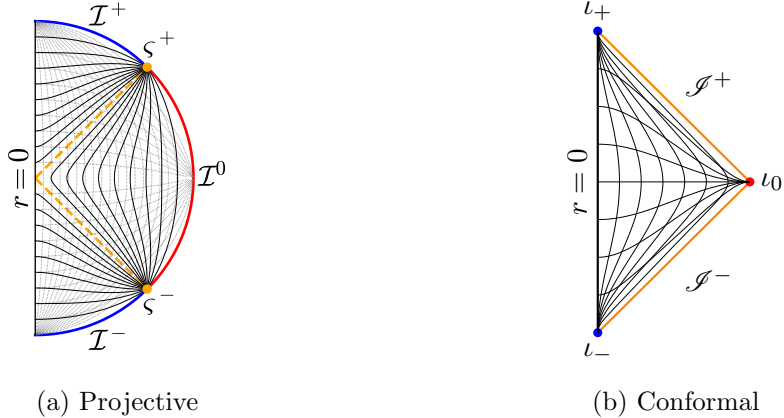
\begin{figure}[h!]
\begin{subfigure}{.5\textwidth}
\centering
\includegraphics[scale=.5]{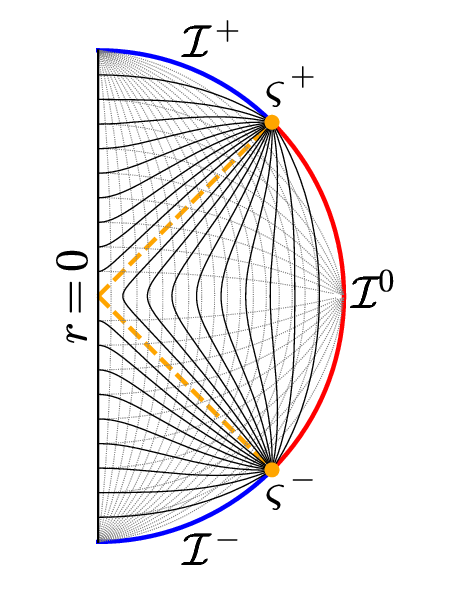}
\subcaption{\label{fig:CProjMinkowski}Projective}
\end{subfigure}%
\begin{subfigure}{.5\textwidth}
\centering
\begin{tikzpicture}[scale=2,thick]
\draw (0,-1) -- node[rotate=90,above]{$r=0$} (0,1) ;
\draw[orange] (0,1) --node[black, above right]{$\scri^+$} (1,0);
\draw[orange] (0,-1) --node[black,below right]{$\scri^-$} (1,0);
\fill[blue] (0,1)  circle(.3mm) node[black,above]{$\iota_+$};
\fill[red] (1,0)  circle(.3mm) node[black,right]{$\iota_0$};
\fill[blue] (0,-1) circle(.3mm) node[black,below]{$\iota_-$};
  \tikzset{declare function={
        T(\t,\r)  = \fpeval{1/pi*(atan(\t+\r) + atan(\t-\r))};
        R(\t,\r)  = \fpeval{1/pi*(atan(\t+\r) - atan(\t-\r))};
    }}
    
     \def\Nlines{3} 
    \newcommand\samp[1]{ tan(90*#1) } 
    \foreach \i [evaluate={\t=\i/(\Nlines+1);}] in {-\Nlines,...,\Nlines}{
        \message{Drawing i=\i...^^J}
        \draw[line width=0.3,samples=30,smooth,variable=\r,domain=0.001:1]
        plot({ R(\samp{\t},\samp{\r}) }, { T(\samp{\t},\samp{\r}) });
    }
    \def\Nlines{6}
  \foreach \i [evaluate={\r=\i/(\Nlines+1);}] in {1,...,\Nlines}{
        \message{Drawing i=\i...^^J}
        \draw[line width=0.3,samples=30,smooth,variable=\t,domain=-0.999:0.999]
        plot({ R(\samp{\t},\samp{\r}) }, { T(\samp{\t},\samp{\r}) });
    }
\end{tikzpicture}
\subcaption{Conformal\label{fig:CConfMinkowski}}
\end{subfigure}
\caption{The projective and conformal compactifications of $\mathbb{M}^{d+1}$.}

\hfill\begin{minipage}{\dimexpr\textwidth-1cm}
 On the left (a), the projective compactification: $\mathcal{I}^{\pm} \simeq \mathbb{H}^d$ denotes future/past projective infinity, isometric to hyperbolic space. $\mathcal{I}^0 \simeq dS^d$ denotes projective spatial infinity, isometric to de Sitter space. Projective null infinities $\varsigma^{\pm}\simeq S^{d-1}$ are (Riemannian) conformal spheres. Altogether the boundary is topologically $S^d$. The lightcone of the origin is represented by the dotted orange line, it divides the compactified spacetime into different regions that are foliated with the level hypersurfaces of $\rho=\frac{1}{\sqrt{\lvert{-t^2+x^2}\rvert}}$. The faint dotted horizontal (resp. vertical) curves are the $t=\textrm{cst}$ (resp. $r=\textrm{cst}$) slices.\\
On the right (b), the conformal compactification: the inside is foliated horizontally (resp. vertically) by the $t=\textrm{cst}$ (resp. $r=\textrm{cst}$) hypersurfaces.  $\mathscr{I}^{\pm} \simeq S^{d-1} \times \mathbb{R}$ denotes future/past conformal null infinity while conformal future/past $\iota^{\pm}$ and spatial $\iota^0$ infinities are reduced to points. The total boundary is topologically $S^1 \times S^{d-1}$.
\end{minipage}~\hspace{0.5cm}
\end{figure}

The \emph{projective} picture outlined above is markedly different from that of the widely known \emph{conformal} compactification of Penrose \cite{penrose_asymptotic_1963}. Conformal compactification has proven to be a valuable conceptual tool that formalised the notion of isolated system in general relativity \cite{geroch_asymptotic_1977,ashtekar_geometry_2015} and in general helped to phrase and attack global problems in mathematical relativity \cite{valiente_kroon_conformal_2016,Friedrich_2018}. The conformal boundary of asymptotically flat spacetime is null, with topology  $\R \times S^2$, and usually referred to as (conformal) null infinity or $\scri$. This, however, differs from projective null infinity, that is, the part of the projective boundary composed of endpoints to null geodesics and which we denoted by $\varsigma^{\pm}$ in Figure \ref{fig:CProjMinkowski}, which is topologically a sphere $S^2$. One can roughly say that the projective construction at null infinity is missing a $1$-dimensional fibre.

On the other hand, in the timelike and spacelike regions (reduced to points in the conformal picture, see Figure~\ref{fig:CConfMinkowski}), the projective boundary is richer than that of the conformal compactification. This feature seems to suggest that one might hope to be able to capture information about their asymptotics and, indeed, there is a rich literature on asymptotic flatness at spatial infinity~\cite{ashtekar_unified_1978,beig_einsteins_1982,Chrusciel:1989ye,ashtekar_spatial_1992,Friedrich:1998-83,Friedrich:1999wk,Kroon:2003ix,Kroon:2004pu,Kroon:2004me,compere_relaxing_2011,Prabhu:2019fsp,AliMohamed:2021nuc,Prabhu:2021cgk,Capone:2022gme,compere_asymptotic_2023,Ashtekar:2024aa,Ashtekar:2023zul,Mohamed:2023jwv}, which uses assumptions closely related to projective compactness \cite{Cap:2014ab}.

However, there is not enough room at $\mathcal{I}^\pm, \mathcal{I}^0$ to encode, for example, asymptotic symmetries in a way which is intrinsic to the boundary and, accordingly, any form of gravitational data in a geometrical manner. In fact, it was already suggested in the seminal work of Ashtekar--Hansen~\cite{ashtekar_unified_1978} to circumvent this problem by introducing a line bundle $\widetilde{\mathcal{I}}^{\epsilon} \simeq \mathcal{I}^{\epsilon} \times \mathbb{R}$ over the projective boundaries $\mathcal{I}^0 \simeq dS^3$, $\mathcal{I}^{\pm} \simeq H^3$ - thus imitating the model of conformal null infinity $\scri^{\pm} \simeq \varsigma^{\pm} \times \mathbb{R}$, thought of as a line bundle over projective null infinity $\varsigma^{\pm} \simeq S^2$. This idea resurfaced in the work of Figueroa-O'Farril and collaborators \cite{Figueroa-OFarrill:2021sxz} where such extensions were naturally constructed as homogeneous spaces. 

The essential feature of the homogeneous model studied by Figueroa-O'Farril, Have, Prohazka and Salzer in \cite{Figueroa-OFarrill:2021sxz} is the following: the action of the Poincaré group $ISO(n-1,1)$ on $\mathbb{R}\text{P}^{n+1}$ fixes a point. Once this point is removed, the decomposition of $\mathbb{R}\text{P}^{n+1}\setminus \{pt\}$ into orbits, given by the top line of \eqref{Introduction: Homogeneous model}, \emph{fibres over the projective compactification}, given by the bottom line of \eqref{Introduction: Homogeneous model}, of $n$-dimensional Minkowski space $\mathbb{M}^n$.
	\begin{equation}\label{Introduction: Homogeneous model}
	\begin{array}{cccccc}
		\mathbb{R}\text{P}^{n+1}\setminus \{pt\} &\simeq& \left( \mathbb{R} \times \mathbb{M}^n\right) &\sqcup&  \left(\textrm{Ti}^{n} \sqcup \mathscr{I}^{n-1} \sqcup  \textrm{Spi}^{n}\right)\\[1em]
		\big\downarrow && \big\downarrow && \big\downarrow\\[1em]
		\mathbb{R}\text{P}^{n}  &\simeq& \mathbb{M}^n &\sqcup&  \left( H^{n-1} \sqcup  S^{n-2} \sqcup  dS^{n-1}\right).
	\end{array}
\end{equation}
Crucially, the action of $ISO(n-1,1)$ on the homogeneous spaces $\textrm{Ti}$, $\textrm{Spi}$ and $\scri$ is non-trivial along the 1-dimensional fibres and this suggests that $\textrm{Ti}$ and $\textrm{Spi}$ should be thought of as \enquote{extended} boundaries.

In previous work~\cite{RSTA20230042}, we studied in detail the curved picture corresponding to this homogeneous model. More precisely, we investigated the curved orbit decomposition of compact $(n+1)$-manifolds $\widetilde{M}$ equipped with a projective structure $[\nabla]$ whose corresponding normal Cartan connection~\cite{Cartan:1924aa,Sharpe:1997aa,Kobayashi:1995aa} admits a holonomy reduction~\cite{cap_holonomy_2014} to a specific Poincaré group described by the existence of two parallel tractors $I^A, H^{AB}$. 
The resulting picture is that $\widetilde{M}$ is naturally endowed with a distinguished $1$-dimensional foliation given by \enquote{integral curves}\footnote{Since we have a densified vector field, we point out that the parametrisation of these curves is undetermined.} of a distinguished densified vector field $n^a$.  Assuming, after excluding a discrete number of singular points $\mathcal{Z}(n^a)$, that the leaf space $\Q$ of the foliation was a manifold, we found that it inherited a projective structure whose normal Cartan connection also admits a holonomy reduction such that the curved orbit decomposition of $\widetilde{M}\setminus \mathcal{Z}(n^a)$ fibers over that of $\Q$ according to the following diagram:
\begin{center}
\begin{tikzcd}[column sep=2mm]
\widetilde{M}\setminus \mathcal{Z}(n^a)\arrow[d] &=&\mathcal{O}\arrow[d] &\sqcup& \overbrace{ \left(\widetilde{\mathcal{H}}_- \sqcup \widetilde{\mathcal{H}}_0 \sqcup \widetilde{\mathcal{H}}_+\right)}^{\partial\widetilde{M}}\arrow[d]\\
\Q  &\simeq& \Qint &\sqcup&  \underbrace{\left(  \mathcal{H}_- \sqcup  \mathcal{H}_0 \sqcup   \mathcal{H}_+\right)}_{\B}
\end{tikzcd}.
\end{center}
This directly generalises\footnote{Note that here $\mathcal{H_+}$ stands for spatial infinity and $\mathcal{H_-}$ for timelike infinity. Thus both $\mathcal{H}_-$ and $\mathcal{H}_0$ are possibly disconnected with past and future components.} the homogeneous model \eqref{Introduction: Homogeneous model}: $\Qint$ splits into two parts $\Qint_{\pm}$ and $\Qint_\pm\cup \B$ is a projective compactification~\cite{Cap:2014ab} of order $1$ of an Einstein metric inherited on $\Qint_{\pm}$. This model therefore provides a geometric situation in which we naturally have a line bundle over the projective compactification, directly generalising \cite{Figueroa-OFarrill:2021sxz}.

 Moreover, as we will recall in Definition~\ref{DefinitionExtendedBoundary}, the line bundles $\widetilde{\mathcal{H}}_\pm \rightarrow \mathcal{H}_\pm$ can in fact be defined in terms of intrinsic data on $\mathcal{H}_\pm$, i.e. only in terms of the data of the projective compactification and without referring to the higher dimensional manifold. What is more, asymptotic symmetries are then naturally identified with automorphisms of $\mathcal{H}_\pm$. Consequently, it seems justified to refer to the total space $\widetilde{\mathcal{H}}_{\pm}$ of the line bundle as the \emph{extended (projective) boundary} of $M$. 

Nevertheless, at this stage, it was unclear if there was any other geometric structure induced by the higher dimensional manifold $\widetilde{M}$ that could in fact be intrinsic to the projectively compact Ricci flat manifolds $M$. The main goal of this article is to provide a \emph{positive} answer to this question\footnote{
	In this work, we are interested in the space/time like parts of the boundary and will work away from the null part, $\mathcal{H}_0$, which will be the object of future work.}: 

\begin{theo*} Given a projectively compact Ricci flat metric $\hat{g}$ on a manifold with boundary $M=\Mint \cup \partial M$, then there is a natural way to equip the extended boundary $\widetilde{\mathcal{H}}^{\pm}  \rightarrow \mathcal{H}^\pm$ with a projective structure that admits a holonomy reduction to strong \emph{Carrollian} geometries. \end{theo*}
 This is the content of Theorem~\ref{Theorem: Cartan geometry on Ti/Spi} and Proposition~\ref{Prop:ProjectiveStructureHolonomyReduction}. 
Here, the term \emph{Carrollian geometries} refers to the degenerate geometries introduced in \cite{duval_carroll_2014}, see \cite{Herfray:2021qmp} for details on some aspects of the related Cartan geometries.\\

The strategy towards obtaining this result is to again consider a curved orbit decomposition, closely related to the homogeneous model \eqref{Introduction: Homogeneous model}, but instead coming from a holonomy reduction of a \emph{non-effective} Cartan geometry on a $n$-dimensional manifold $M$ based on non-effective realisation of projective space\footnote{Here $\textrm{Aff}(n)= SL(n) \ltimes \mathbb{R}^{n}$ is the affine group, so that $\R P^n \simeq \frac{(G/K)}{(P/K)}\approx \frac{SL(n+1)}{\textrm{Aff(n)}}$ is the usual, effective, realisation of projective space.}
\begin{equation}\label{Introduction: non effective projective model}
	\R \textrm{P}^n \simeq G/P \approx \frac{SL(n+1)\ltimes \mathbb{R}^{n+1}}{\textrm{Aff}(n)\ltimes \mathbb{R}^{n+1}}.
\end{equation}

The details of this realisation, as well as the corresponding Cartan geometry and its holonomy reduction will be discussed in Sections~\ref{ssec:NoneffectiveModel} and \ref{ssec:metric-holonomy-reduction}.

 One can always recover an effective Cartan geometry from a non-effective one by quotienting by the kernel $K$ (in the case at stake $K\approx\R^{n+1}$). In general, trying to use a non-effective geometry, such as \eqref{Introduction: non effective projective model}, to study the underlying effective geometry is nevertheless likely to be a futile task, as it is not clear that the extra degrees of freedom in the kernel will carry any meaningful intrinsic information. However, in the case at hand we will prove that the holonomy reduction will completely fix the non-effective geometry in terms of the effective one.

More precisely, we will prove (see Theorem~\ref{Theorem: UniquenessOnInside}) that, in the interior $\Mint$, the non-effective Cartan geometry, modelled on \eqref{Introduction: non effective projective model}, is completely determined by the projectively compact metric metric $\pg$. We then prove that it extends to points of the usual projective boundary $\mathcal{H}^{\pm}$, see Theorem~\ref{Thm:ExtensionToBoundary}. Finally, we prove that this non-effective Cartan geometry, canonically induced on $\mathcal{H}^{\pm}$ by the projective compactification, is equivalent to an effective Carrollian geometry on the extended boundary $\widetilde{\mathcal{H}}^{\pm}$.

The paper is organised as follows: in Section \ref{sec:Projective Tractor Calculus} we review elements of projective tractor calculus and related results for projectively compact metrics that we will need on the rest of this article. We then discuss in Section \ref{sec:HolonomyReductionOfNonEffective} the curved orbit decomposition induced by the holonomy reduction of a non-effective Cartan geometry modelled on \eqref{Introduction: non effective projective model}. Finally, in Section \ref{sec:Projective structure on the extended boundary}, we derive the geometry induced on the extended boundary $\widetilde{H}^{\pm}$ from this construction.

\section*{Notations}
We will make extensive use of the abstract index notation. In this context, we will denote the sheaf of vector fields, $\E^a$, the sheaf of one forms, $\E_a$, etc.
To simplify notations, when the context is clear, we will abuse notation and not specify over which open set we are working, or, only specify the manifold, e.g. $\E^a_{\M}$.

The bundle of projective densities of weight $w$ (defined below) will be written $\E(w)$. Once more, for readability, we will abuse notation and not distinguish between $\E(w)$ and the modules of local sections; which interpretation is appropriate should always be clear from the context. In cases where confusion could arise, will use the notation $\Gamma$ to indicate sections. 
\section*{Acknowledgements}
JB and YH are happy to thank Andreas \v{C}ap for discussions related to the content of this article. JB gratefully acknowledges that this research is supported by NSERC Discovery Grant 105490-2018. Finally, the authors are thankful to the Erwin Schrodinger Institute for welcoming them during the workshop  ``Carrollian Physics and Holography''(2024) where part of this manuscript was written.

\section{Brief review of the standard projective tractor calculus}\label{sec:Projective Tractor Calculus}
For the convenience of the reader we review elements of the usual projective tractor calculus, developed in~\cite{Bailey:1994aa}.
We then recall from \cite{Cap:2014aa,Cap:2014ab,cap_projective_2016-1} Čap-Gover's definition for projectively compact Ricci flat manifold and some associated results.

\subsection{Basic definitions}\label{BasicDefinitions}
Let $\M$ be a manifold\footnote{With or without boundary, we refer the reader, for instance, to~\cite{Borthwick:2023aa} for a detailed discussion about defining projective structures on manifolds with boundary.} of dimension $n$ endowed with an equivalence class of projectively equivalent torsion-free connections $[\pnabla]$. It is a standard result\footnote{See, for instance,~\cite[Proposition 7.2]{Kobayashi:1995aa}} that two torsion-free affine connections $\tilde\nabla$ and $\nabla$ are projectively equivalent if and only if there is a 1-form $\Upsilon_a$ such that for any vector field $\xi^b$:
\begin{equation}\label{ProjectiveChangeOfConnectionVectors}\tilde\nabla_a\xi^b=\nabla_a\xi^b + \Upsilon_a\xi^b +\Upsilon_c\xi^c\delta_a^b. \end{equation}
We will use $\tilde\nabla=\nabla + \Upsilon$ as shorthand for this condition.

We denote by $\mathcal{E}(w)$ the associated vector bundle to the frame bundle $P^1(\M)$ determined by the $1$-dimensional representation of $GL_{n}(\R)$: $A \mapsto \lvert \det {A}\rvert ^{\frac{w}{n+1}}.$
Sections of this bundle will be called \emph{projective densities} of weight $w$. Their definition guarantees that under the projective change of connection $\bnabla = \pnabla + \Upsilon$ we have, for any section $\sigma \in \mathcal{E}(w)$, \[\bnabla\sigma = \pnabla \sigma + w\Upsilon \sigma. \]
For any vector bundle $\mathcal{B}$ with base $\M$, we will write: $\mathcal{B}(w)=\mathcal{B}\otimes \mathcal{E}(w)$ and we will say that sections of $\mathcal{B}(w)$ are of weight $w$.

Let $G=\textnormal{PGL}(n,\R)$ and $H=(\R^{n})^*\rtimes GL(n,\R)$ the isotropy subgroup of a given line. The class $[\nabla]$ determines a reduction of the second-order frame bundle $P^2(M)$ to a $H$-principle bundle $P$, on which there is a unique \emph{normal} Cartan connection $\omega$~\cite{Cartan:1924aa, Kobayashi:1995aa}; this is a normal Cartan geometry modelled on projective geometry as described in~\cite{Sharpe:1997aa}.

The \emph{standard dual (projective) tractor bundle}, is defined~\cite{Bailey:1994aa} to be the $1$-jet prolongation of $\E(1)$:
\[\T^*=J^1\E(1).\]
In the abstract index notation sections of tractor bundles will be indicated by indices $A,B,C\dots$.

$\T^*$ fits into a canonical short exact sequence of vector bundles:
\begin{center}
\begin{tikzcd} 0 \arrow[r] & T^*\M(1) \arrow[r,"Z"] & \arrow[l,dotted,bend left = 45, "W"] \mathcal{T}^* \arrow[r,"X"] & \E(1) \arrow[r] \arrow[l,dotted,bend left=45,"Y"] &0. \end{tikzcd}
\end{center}
A choice of connection in the projective class $[\nabla]$ provides a non-canonical isomorphism $\mathcal{T}^*\simeq T\M(1)\oplus \E(1)$, described above by the non-canonical maps $Y$ and $W$. When thinking of these maps as sections of $\mathcal{T}^*\otimes\mathcal{E}(-1)$ and $T^*\M(1)\otimes\mathcal{T}$ respectively, they transform under change of connection $\tilde{\nabla}=\nabla + \Upsilon$ according to:
\begin{align}\label{ChangeOfProjectiveSplitting}\tilde{Y}_A&=Y_A - \Upsilon_aZ^a_A, & \tilde{W}^A_a&=W^A_a + \Upsilon_a X^A. \end{align}
Fixing a choice of connection $\nabla$ in the projective class, sections of $\T^*$ will be written:
\[ T_A= \sigma Y_A+\mu_aZ^a_A ,\]
similar expressions will be used with sections of tensor powers of the tractor bundle. 

We also recall the standard decomposition of the curvature tensor $R_{ab\phantom{c}d}^{\phantom{ab}c}$:
\[ R_{ab\phantom{c}d}^{\phantom{ab}c}=W_{ab\phantom{c}d}^{\phantom{ab}c} + 2\delta^c_{[a}P_{b]d} + \beta_{ab}\delta^c_d,\]
where the Weyl tensor $W$ is trace-free, $P_{ab}$ is the projective Schouten tensor and  $\beta_{ab}$ can be thought of as the curvature tensor of the density bundle $\mathcal{E}(1)$. They are given by
\begin{align*}
	(n-1)P_{ab}&= R_{ab}+\beta_{ab},& \beta_{ab}&=-\frac{2}{n+1}R_{[ab]}=-P_{[ab]},
\end{align*}
where $R_{ab}=R_{da\phantom{d}b}^{\phantom{da}d}$ is the Ricci tensor. We note that the tensor quantities $P$ and $\beta$ are not projectively invariant. In particular, under a projective change of connection $\tilde{\nabla}=\nabla +\Upsilon$:
\begin{equation}\label{SchoutenTensorTransformationRule}  \tilde{P}_{ab}=P_{ab} -\nabla_a \Upsilon_b + \Upsilon_a \Upsilon_b. \end{equation}

A connection $\nabla$ is said to be \emph{special} if it preserves a nowhere vanishing density $\sigma$. In this case $\beta_{ab}=0$ (in particular all density bundles are flat) and we say that $\nabla$ is the \emph{scale} determined by $\sigma$. Even when it does not preserve a nowhere vanishing density, we will refer to a choice of connection in the class as a choice of scale.

The normal Cartan connection $\omega$ induces a linear connection on $\T^*$. Having fixed a projective connection $\nabla$ in the class, and split the short-exact sequence, the action of this connection can be summarised conveniently by the relations:
\begin{equation}\label{TractorConnection} \nabla_a X^A= W^A_a, \quad \nabla_a W^B_b=-P_{ab}X^B,\quad \nabla_a Y_A=P_{ab}Z^b_A, \quad \nabla_a Z^b_B=-\delta_{a}^bY_B. \end{equation}

The tractor curvature, defined by the identity $\Omega_{ab\phantom{C}D}^{\phantom{ab}C}T^D= 2 \nabla_{[a}\nabla_{b]}T^C$, is known to be given in an arbitrary scale by:
\begin{equation}\label{TractorCurvature} \Omega_{ab\phantom{C}D}^{\phantom{ab}C}=\Weyl{a}{b}{c}{d}W^C_cZ^d_D -Y_{abd}X_CZ^d_D \end{equation}
where: $Y_{abd}=2\nabla_{[a}P_{b]d}$ is the projective Cotton tensor.

For any weighted tractor bundle, $T^\mathscr{A}(w)$, where $\mathscr{A}$ is an arbitrary set of abstract indices, the normal Cartan connection induces a natural projectively invariant differential operator:
\[ D: T^\mathscr{A}(w) \rightarrow \T_A\otimes T^\mathscr{A}(w-1)\]
known as the Thomas $D$-operator and defined by:
\begin{equation}\label{Definition: Thomas operator}
D_A t^\mathscr{A} = w\,t^\mathscr{A}\,Y_A + \nabla_a t^\mathscr{A}\,Z^a_A.
\end{equation} 
$D$ behaves much like a covariant derivative with a cotractor index, in particular, \emph{it satisfies the Leibniz rule}; this can be useful in computations. Also observe that one has
\begin{equation}\label{DderivativeX} D_AX^B = \delta_{\st A}^{B}. \end{equation}

\subsection{Metric projective structures}
A projective class $[\nabla]$ on an $n$-manifold $\M$ is said to be \emph{metric} if it admits a solution $\zeta^{ab} \in \E^{(ab)}(-2)$ 
to the Mike\v{s}-Sinjukov metrisability equation~\cite{Mikes1996,Sinjukov:1979aa}:
\begin{equation}\label{Metrisability} 
	\nabla_c\zeta^{ab} - \frac{2}{n+1}\delta^{(a}_c\nabla_d\zeta^{b)d}=0. 
\end{equation}
If we introduce\footnote{$\epsilon^2$ is the section that realises the identification $(\Lambda^{n}T\M)^{\otimes 2} \simeq \mathcal{E}(2(n+1))$; if $\M$ is oriented, then it can be thought of as the square of the volume form.} 
	\begin{equation*}
	\detd : \begin{array}{ccc}
	\mathcal{E}^{ab}(w) & \longrightarrow &\mathcal{E}[n(w+2) + 2 ] \\
	h^{ab} &\longmapsto& \frac{1}{n!}\epsilon^2_{a_1 ... a_{n}b_1... b_{n}} h^{a_1b_1}... h^{a_{n}b_{n}}
	\end{array}
	\end{equation*}then
solutions of Eq.~\eqref{Metrisability} yield inverse metrics $g^{ab}=\lvert\detd(\zeta)\rvert \zeta^{ab}$ on $\{\pm\detd(\zeta)>0\}$ whose Levi-Civita connection $\nabla^g$ belongs to the projective class $[\nabla]$. $\nabla^g$ is equivalently characterised as the special connection that preserves the density $\detd(\zeta)$.

Solutions of the metrisability equation been shown~\cite{Eastwood:2008aa,Flood:2018aa} to be in one-to-one correspondence with tractors $H^{AB}$ that satisfy:
\begin{equation} \label{TractorMetrisability}
	\nabla_c H^{AB} + \frac{2}{n}X^{(A}W_{cE\phantom{B}F}^{\phantom{cE}B)}H^{EF}=0,
\end{equation}
where $W_{cE\phantom{B}F}^{\phantom{cE}B}=Z^c_C\Omega_{ce\phantom{B}F}^{\phantom{ce}B}$ and $\Omega_{ce\phantom{B}F}^{\phantom{ce}B}$ is the tractor curvature.
In this case $H^{AB}$ is of the form\footnote{In what follows we will have $\lambda^a = - \pscale N^a$.}
\begin{equation} \label{ProjectiveTractorMetric}H^{AB}= \zeta^{ab}\,W^A_aW^B_b  +\lambda^b\,X^{(A}W^{B)}_b+ \nu \,X^AX^B,\end{equation}
with 
\begin{align}\label{H decomposition}
\lambda^a &= -\frac{2}{n+1}\nabla_c\zeta^{ca},& \nu&=  \frac{\nabla_a\nabla_b \zeta^{ab}}{n(n+1)}+\frac{1}{n}P_{ab}\zeta^{ab}.
\end{align}

If $\nabla_c H^{AB}=0$, then it can be shown that Eq.~\eqref{TractorMetrisability} is automatically satisfied; these solutions are said to be \emph{normal}. This is equivalent to
\begin{align}\label{DH=0}
\nabla_c \zeta^{ab} + 2 \delta^{(a}_c \lambda^{b)} &=0, &
\nabla_c \lambda^a + \delta^a_c \nu - P_{cb}\zeta^{ba} &=0.
\end{align}
(and implies \eqref{H decomposition}). Normal solutions are known to correspond to Einstein metrics~\cite{Cap:2014aa}. Introducing\footnote{We recall from \cite{Flood:2018aa} that the projective tractor volume form is defined by \begin{equation*} \epsilon^2_{A_0 A_1 ... A_{n}B_1 B_1... B_{n}} :=  \epsilon^2_{a_1 ... a_{n}b_1... b_{n}} Y_{[A_0} Z_{A_1}^{a_1} ...  Z_{A_{n}]}^{a_{n}} Y_{[B_0} Z_{B_1}^{b_1} ...  Z_{B_{n}]}^{b_{n}}
	\end{equation*}}
\begin{equation*}
	\det : \begin{array}{ccl}
		\T^{AB} & \longrightarrow & C^{\infty}(\M) \\
		H^{AB} &\longmapsto& \frac{(n+1)}{n!}\epsilon^2_{A_0 A_1 ... A_{n}B_0 B_1... B_{n}} H^{A_0B_0}... H^{A_{n}B_{n}}
	\end{array}
\end{equation*}
then $\lvert\detd(\zeta)\lvert\zeta^{ab} \in \mathcal{E}^{ab}$ is Einstein of scalar curvature
\begin{equation}
R = n(n-1) \sgn(\detd \zeta) \det H.
\end{equation}
In the scale associated to $\lvert\detd(\zeta)\lvert$ this can be rewritten as $R = n(n-1) \lvert\detd \zeta\rvert  \nu$.

\subsection{Projective compactification}\label{ClassicalProjectiveCompactification}
The notion of projective compactification was first introduced in the work of Čap-Gover~\cite{Cap:2014aa,Cap:2014ab,cap_projective_2016-1}. We recall that an affine connection $\pnabla$ on the interior $\Mint$ of a manifold with boundary $\M$, is said to be \emph{projectively compact of order $\alpha$} if each boundary point $x\in \B$ admits a neighbourhood $x\in U \subset M$ and a boundary defining function $\rho: U \rightarrow \R_+$, such that the projectively equivalent connection $\bnabla=\pnabla + \frac{\textrm{d}\rho}{\alpha \rho}$ on $U\cap \Mint$ admits a smooth extension to $U$. By extension, a metric $\pg$ on $\Mint$ is said to be projectively compact of order $\alpha$, if its Levi-Civita connection, $\pnabla$, is projectively compact in this sense.

In this article, our main focus is the case of a \emph{Ricci flat} projectively compact \emph{Lorentzian}\footnote{Signature $(n,1)$, for definiteness, as our discussion is essentially independent of this choice.} metric $\pg$, therefore $\alpha=1$~\cite[Theorem 3.3]{Cap:2014ab}. As discussed in~\cite{cap_holonomy_2014,Cap:2014ab,Flood:2018aa} the boundary $\B$ is then totally geodesic and inherits a projective structure which admits a holonomy reduction to the Poincaré group. This gives a curved orbit decomposition, generalising the flat model depicted in Figure \ref{fig:CProjMinkowski}:
\begin{equation}\label{BoundaryOrbitDecomposition} \B=\mathcal{H}_{\pm}\cup \mathcal{H}_0,\end{equation}
where $\mathcal{H}_{\pm}$ are closed hypersurfaces that inherit Einstein metrics which are projectively compact of order $2$: spatial infinity $\mathcal{H}_{+}$ is a Lorentzian manifold with positive scalar curvature while past/future time infinity $\mathcal{H}_{-}$ is Riemannian with negative scalar curvature. The projective boundary for these metrics coincides with projective null infinity $\mathcal{H}_0$; it is a $(n-2)$-dimensional submanifold inheriting a conformal structure.

We will review here from these references some essential results. We will make extensive use of these in the rest of the articles and this will also be useful to fix notations.\\

Define the density and densified inverse metric\footnote{It then follows from this definition that $\lvert\detd(\zeta^{ab})\lvert = \pscale^2$ and therefore the data of $\zeta^{ab}$ is strictly equivalent to that of $\pg^{ab}$.}: 
\[ \pscale=|\textrm{Vol}_{\pg}|^{\frac{1}{n+1}},\qquad  \zeta^{ab}=\pscale^{-2}\pg^{ab},\] 
where $\pscale$ is extended by $0$ to $\M$ and is a \emph{canonical boundary defining density} in the sense of~\cite[Proposition 2.3]{Cap:2014ab}, in particular:
\[ \pscale =0,\qquad  \nabla_a\pscale\neq0, \qquad \textrm{along $\B$}. \]
On the other hand, $\zeta^{ab}$ is clearly a normal solution to the metrisability equation~\eqref{Metrisability}.

From the results in~\cite{Cap:2014ab}, Ricci flatness implies that $\zeta^{ab}$ and $\pscale$ define tractors $H^{AB}$ and $I_B$ such that:
\begin{align}\label{ProjectiveTractor: H and I} \nabla_cH^{AB}&=0,& \nabla_c I_B&=0,& I_B&=D_B\pscale,& H^{AB}I_A&=0.\end{align}
Since the projective structure extends to $M$, $H^{AB}$ and $I_B$ extend to parallel tractors on $M$. 

From~\eqref{ProjectiveTractorMetric}, \eqref{Definition: Thomas operator}, the expression of $H^{AB}$ and $I_A$ in the Levi-Civita scale $\pnabla$ is readily seen to be simply: 
\begin{align}
H^{AB}&=\zeta^{ab} \hat W_a^A \hat W_b^B,  & I_A & = \pscale \hat Y_A.
\end{align}
It follows then from the transformation rules~\eqref{ChangeOfProjectiveSplitting} that, since any other scale $\bnabla$ is related to $\pnabla$ by \[\bnabla=\pnabla+\pscale^{-1}\bnabla_a\pscale,\]
we have $I_B = \pscale Y_A + \nabla_a \pscale Z_A^a$ and
\begin{align}\label{TractorMetricBoundaryScale} H^{AB}&= \zeta^{ab}W_a^AW_b^B -2\pscale^{-1}\zeta^{ab}\bnabla_b\pscale W^{(A}_{a}X^{B)}+\pscale^{-2}\zeta^{ab}\bnabla_a\pscale\bnabla_b\pscale X^AX^B.
 \end{align}
 Comparing with \eqref{ProjectiveTractorMetric} gives expressions for $\lambda^a$ and $\nu$. If $\bnabla$ extends to the boundary then the components of $H^{AB}$ must also have smooth extensions to $\B$ and therefore $\zeta^{ab}$, $\lambda^a$, $\nu$  must extend to boundary points. In particular
\[\zeta^{ab}\bnabla_a \pscale =0,\qquad  \pscale^{-1}\zeta^{ab}\bnabla_a\pscale\bnabla_b\pscale=0, \qquad \textrm{along $\partial M$}, \]
for any connection $\bnabla$ that extends to the boundary. It follows that the restriction of $\zeta^{ab}$ to $\B$ defines a canonical tensor \begin{equation}\label{Definition: boundary metric}
\bdh^{ab} \in \E^{(ab)}_{\B}(-2)
\end{equation}
which is itself a solution to the metrisability equation~\eqref{Metrisability} on $\B$, with associated tractor $\bar{H}^{AB}$. A key point is that it gives a canonical section\footnote{Due to different possible signatures, $\dscale$ is positive on $\mathcal{H}_{+}$ and negative on $\mathcal{H}_{-}$.} $$\dscale :=-\detd \bdh^{ab} \in \E(2)_{\B}$$ of the intrinsic density bundle on $\B$ that provides the orbit decomposition \eqref{BoundaryOrbitDecomposition}. In particular, it is a boundary defining density for $\mathcal{H}_0$, 
$\mathcal{H}_0 = \mathcal{Z}(\lambda_0)$, and has a fixed sign on $\mathcal{H}_{\pm}$:
\begin{align*}
 \dscale &=0 \qquad \text{on}\; \mathcal{H}_0, & \dscale &=\pm|\dscale| \qquad \text{on}\; \mathcal{H}_{\pm}.
\end{align*}

 As was proved in \cite{Cap:2014ab}, near points of $\mathcal{H}_\pm$, one can in fact always choose a scale such that both
 \begin{align}\label{Projective compactification review: N and nu def}
	N^a &:=  \pscale^{-2}\zeta^{ab}\bnabla_b\pscale, & \nu &:= \pscale^{-2}\zeta^{ab}\bnabla_a\pscale\bnabla_b\pscale,
\end{align}
are finite along the boundary. We reformulate this fact in:

\begin{theo}\mbox{}\label{Projective compactification review: CapGoverBoundaryDefiningFunction}
A \underline{special} scale $\bnabla$, defined on a neighbourhood $U$ of a point $x_0 \in \mathcal{H}_\pm$, is such that 
\[ N^a: =\pscale^{-2}\zeta^{ab}\nabla_b \pscale \]
extends smoothly to $\B\cap U$ if and only if it preserves a density $\tau$ satisfying
\[ i^*\tau = \sqrt{|\lambda_0|}.\] 

In this case, the canonical density $\lambda_0$ is given indifferently on $U$ by any of the following,
	\begin{align}\label{Projective compactification review: identities on lambda nu}
		\dscale = -\detd \bdh^{ab} = i^* \nu^{-1} = \pm i^* \tau^2 \qquad \in\E_{\B}(2). 
	\end{align} 
When evaluated in such scales the metric can then be rewritten as
\begin{equation}\label{Projective compactification review: asymptotic form of the metric} \zeta_{ab}= \frac{\nu^{-1}}{\pscale^2}\nabla_a\pscale \nabla_b\pscale + q_{ab},
		\end{equation}
		for some finite tensorial density $q_{ab}$.
\end{theo}
\begin{proof}
First, let $\nabla$ preserve a density $\tau$ and satisfy the condition on $N^a$. Then it follows from results in \cite{Cap:2014ab} (see in particular Eq.~\eqref{Projective compactification review: identities on schouten} below) that $\tau^2\nu$ is constant along the boundary and  (rescaling $\tau$ if necessary by a constant factor) we can achieve $\iota^*\tau^2=\pm i^*\nu^{-1}$. By definition of $N^a$ and $\nu$, it also follows that in this scale
\begin{equation*}
\zeta^{ab}= \nu^{-1} \pscale^2 N^a N^b + q^{ab},
\end{equation*}
with $q^{ab}$ restricting to $\bar h^{ab}$ along the boundary and in particular $q^{ab}\nabla_{b}\pscale =0$. We deduce, since $\zeta^{ab}$ has Lorentzian signature,
\begin{equation*}
\pscale^2 := \lvert\detd\zeta\rvert = - \detd\zeta = -\nu^{-1} \pscale^2 (N^a \nabla_{a}\pscale)^2 \detd \bar{h} + O(\pscale)
\end{equation*}
and thus $\iota^* \nu^{-1} = -\detd \bar{h}$.

Assume now that $i^*\tau = \dscale[1]$. By~\cite[Lemma 3.13, Proposition 3.14]{Cap:2014ab} near every point $x_0 \in \mathcal{H}_\pm$ there exists a density $\tilde\tau$ such that $\tilde{N}^a=\pscale^{-2}\zeta^{ab}\tilde{\nabla}_b\pscale$ extends smoothly to $U\cap \B$, where $\tilde\nabla$ is the  projective scale determined by $\tilde\tau$.
Near $x_0$, one may write: $\tau=\tilde{\tau} + \pscale \omega$ so that
\[\pscale^{-2}\zeta^{ab}\nabla_b\pscale= \tilde{N}^a - \tau^{-1}\zeta^{ab}\tilde{\nabla}_b\omega - \tau^{-1}\omega \pscale \tilde{N}^a.  \]
Since the right-hand side is finite by definition of $\tilde{\tau}$, it follows that the left-hand side is also, and consequently, $N^a$ also extends smoothly to boundary points near $x_0$.
\end{proof}

When using these particularly important scales, the Einstein equations also imply the following useful properties on the projective Schouten tensor:

\begin{align}\label{Projective compactification review: identities on schouten}
	\begin{aligned}
		\nabla_a \nu &= -2\pscale N^bP_{ba} = -\pscale \,\big(2 \nu^{-1} P_{bc}N^bN^c\big)\, \nabla_a \pscale + O(\pscale),
		\\[0.4em] 
		P_{ab}& = \nu\,\bdh_{ab} +  \big(\nu^{-2}P_{cd}N^cN^d\big)\, \nabla_a \pscale \nabla_b \pscale + O(\pscale)
	\end{aligned}
\end{align}

\begin{proof}\mbox{}
	
	 We briefly justify these crucial statements. The first equation is immediate from $\nabla_c H^{AB}=0$ and~\eqref{TractorConnection}. The second follows from~\eqref{DH=0}:
\[P_{cb}\zeta^{ba}= \delta^a_c\nu - \nabla_c\pscale N^a + \pscale \nabla_c N^a. \]
Indeed, decomposing $P_{cb}$:
\[P_{cb}=\nu^{-2}P_{cd}N^cN^d\nabla_a\pscale\nabla_b\pscale +2\nu^{-1}v_{(a}\nabla_{b)}\pscale + \tilde{P}_{ab}\]
where $N^a\tilde{P}_{ab}=0$ and $N^av_a =0$,$\upsilon_a \in \E_a(-3)$.
Since we have the block-diagonal decomposition of the inverse metric:
\[\zeta^{ab}= \pscale^2\nu^{-1} N^a N^b + q^{ab},\]
where $q^{ab}\nabla_b \pscale =0$, Eq.~\eqref{DH=0} shows that:
\[ \nu \delta^{a}_c -\nabla_c\pscale N^a= \nu^{-1}q^{ab}v_b\nabla_a \pscale + q^{ab}\tilde{P}_{bc} + O(\pscale). \]
The first term on the right-hand side is off-diagonal and is consequently vanishing, $q^{ab}v_b \to 0$, which implies that $v_b \to 0$ because $q$ is invertible on the restriction to $T\B$.
\end{proof}

We have already seen that the weight-1 density $\dscale[1] \in \E_{\B}(1)$ is very useful in this context. However, and despite the fact that $\nabla_a\sigma$ provides a natural identification \begin{equation}\label{RestrictionsOfDensitiesAreIdentifiedWithDensitiesOnBoundary}
	\E_\B(1)\simeq \iota^*\E_M(1),
	\end{equation} this boundary density  $\dscale[1]$ has no canonical extension to a section $|\lambda|^{\frac{1}{2}}$ of $\E_M(1)$. This remark was the basis for our definition of $\Spi$ and $\Ti$ in~\cite{RSTA20230042}: these are the total space of line bundles,
	\begin{align*}
		\Spi &\to \mathcal{H}_{+}, & \Ti &\to \mathcal{H}_{-},
	\end{align*}
	 which are tailored to deal with this ambiguity.

Since the Thomas derivative 
of $|\dscale|^\frac{1}{2}$ provides a map \[\mathcal{H}_{\pm} \longrightarrow \T^*_\B\simeq  i^*J^1\E(1)_\M/j^1\pscale=\T^*/I_A,\] we can indeed define the following.
\begin{defi}\label{DefinitionExtendedBoundary}
The \textbf{extended boundaries} $\Spi/ \Ti$, that we will note uniformly $\widetilde{\mathcal{H}}_{+}/\widetilde{\mathcal{H}}_{-}$, are defined to be the pullback to $\mathcal{H}_{\pm}$ of $i^*J^1\E(1)_M=i^*\T^*_M$ (viewed as a bundle with base $i^*J^1\E(1)_\M/J^1\pscale\simeq \T^*_\B$) along the Thomas derivative $D_A \dscale[1]$: 
\begin{center}
\begin{tikzcd}
\eB\quad  \arrow[d, "\pi"] \arrow[r,dashed] &\quad  i^*\T^*_\M \arrow[d,"\textnormal{Can. Proj.}"] \\[0.7em] \mathcal{H}_{\pm}\quad  \arrow[r, "D_A{\dscale[1]}"]&\quad \T^*_\B\simeq \T^*_\M/I
\end{tikzcd}
\end{center}
These are principal $\R$-bundles over $\mathcal{H}_{\pm}$, with $\R$-action defined by
\[ (D_A|\lambda|^{\frac{1}{2}}, k) \mapsto D_A|\lambda|^{\frac{1}{2}} \mp k I_A.\]
\end{defi}

In practice, it can be enlightening to think of sections of this bundle as formal asymptotic expansions to order $1$ in $\pscale$ for an extension of $\dscale[1]$ to a 1-density on $M$. In other words a formal expression of the type:
\[ |\lambda|^\frac{1}{2}= \dscale[1] \mp u \pscale + O(\pscale^2),  \]
where $u$ is a function on $\partial M$.

\section{Non-effective Cartan projective geometry and the extended tractor calculus}\label{sec:HolonomyReductionOfNonEffective}
\subsection{Non-effective homogeneous model} \label{ssec:NoneffectiveModel}
The starting point for our discussion is the Cartan geometry modelled on the \emph{non-effective} description of projective space in $n$ dimensions $\mathbb{R}P^{n}$ as the set of all lines passing through a fixed point $[I]\in \mathbb{R}P^{n+1}$ in projective space in $n+1$ dimensions. Here $I \in \mathbb{R}^{n+2}$ is a fixed vector and this leads to the description of projective space as a homogeneous space
\begin{equation}\label{Homogeneous space model}
	\mathbb{R}P^{n} = G/P
\end{equation}
where $G$ and $P$ are the following Lie subgroups of $PSL(n+2)$:
\begin{align}\label{Non effective model: G and P groups}
\begin{aligned} G&= \left\{ \begin{pmatrix} 1 & \chi_\mu & \chi_0 \\ 0 & A^\mu_\nu & \omega^\mu \\ 0 & \Upsilon_\nu & a  \end{pmatrix} \!\!\!\!\mod Z(SL(n+2)) \right\},\, &P&=\Big\{g \in G, \omega^\mu=0 \Big\}. \end{aligned}
\end{align}
Contrary to the usual description of projective space, the action of $G$ on $\R P^{n}$ is not faithful and its kernel (the largest normal subgroup of $G$ in $P$) is given by the normal subgroup:
 \[K= \left\{ \begin{pmatrix}  1 & \chi_\mu & \chi_0 \\ 0 & \varepsilon \delta^\mu_\nu & 0 \\ 0 & 0&  \varepsilon  \end{pmatrix} \!\!\!\mod Z(SL(n+2)),\quad  \varepsilon I_{n+1} \in Z(SL(n+1)) \right\},\]
we denote $\mathfrak{g},\mathfrak{p}, \mathfrak{k}$ respectively the Lie algebras of the groups $G,P,K$. 
 
The non-effectivity we consider here should be distinguished from that, for instance, present in spin geometry, as the kernel is large enough to be visible in the infinitesimal structure; $K$ is not a discrete group.

Let $\Q$ be a $n$-dimensional manifold equipped with a Cartan geometry ($\mathscr{C} \to\Q, \omega)$ modelled on $G/P$.
A general $\mathfrak{g}$-valued one-form can be parametrised as follows:
\begin{equation}\label{GeneralParametrisationConnection} \begin{pmatrix}0 & \upsilon_\nu& \upsilon_0\\ 0 & \varpi^\mu{}_\nu - \frac{1}{n+1}\varpi^\rho_{\,\,\rho} \delta^\mu_{\,\,\nu}  & \theta^\mu \\ 0 & P_\nu & -\frac{1}{n+1}\varpi^\rho{}_\rho \end{pmatrix}. \end{equation}

We define $\Tscr \to \Q$  the associated vector bundle associated to the restriction to $P$ of the fundamental representation of $PSL(n+2,\R)$. In the abstract index notation, we will use the notation $\ei{A},\ei{B},\dots$ to indicate sections of $\Tscr$.

Quotienting out the kernel $K$, the Cartan geometry ($\mathscr{C}\to\Q, \omega)$ induces a projective connection on $\Q$. We shall make the assumption that the connection obtained from this procedure is \emph{normal} and write $[\nabla]$ for the corresponding projective class; observe that this implies that the Cartan connection $\omega$ is also torsion-free. By extension, we shall also say that these connections are \emph{normal}. 

The bundle $\Tscr$ has the decomposition structure given in Figure~\ref{DecompositionTscr}, realising it as an extension of the usual projective tractor bundle $\T$ with its normal Cartan connection.
\begin{figure}[h!]
\centering
\begin{tikzcd} 0\arrow[r] &\R \arrow[r, "I"] & \Tscr \arrow[r,"\Pi"] & \T \arrow[r] &0.  \end{tikzcd}
\caption{Decomposition sequence of $\Tscr$\label{DecompositionTscr}}
\end{figure}
\FloatBarrier
This filtration corresponds to the existence of distinguished parallel sections $I^{\ei{A}}$, $\Pi_{\,\,\ei{A}}^{A}$,
\begin{align}\label{Derivative of I and Pi}
	\nabla_a I^{\ei{A}} &=0, &\nabla_a \Pi_{\,\,\ei{A}}^{A}&=0
\end{align}
 respectively realising the first injection and the projection. The sequence can be split by any choice of section $L_{\ei{A}}$ such that $L_{\ei{C}}I^{\ei{C}}=1$; two such sections differ by a projective cotractor $\chi_A = \chi_a\, Z^a_A  +  \chi_0\, Y_A$ and each choice realises a non-canonical isomorphism,
\begin{equation}\label{Non effective model: ExteT = R + T}
\Tscr \simeq \R \oplus \T.
\end{equation}
The parameters $\chi_{\mu}$ and $\chi_{0}$ of the realisation \eqref{Non effective model: G and P groups} of the group $G$, correspond to cycling through different choices of splitting or, equivalently, to cycling through different realisations of the isomorphism \eqref{Non effective model: ExteT = R + T}. As this degree of freedom is at this point abstract, for the remainder of this section we shall simply assume the existence of such splittings.

Similarly to the usual tractor calculus, we define the non-canonical splitting operators $\rotpi^{\ei{A}}{}_A$ and $ L_{\ei{A}}$ so that sections of $\Tscr$ and $\Tscr^*$ can be written 
\begin{align}
T^{\ei{A}} &= f \,I^{\ei{A}} + \rotpi^{\ei{A}}_{\,\,A} \,t^A, &
T_{\ei{A}}& = g \,L_{\ei{A}} + \Pi_{\,\,\ei{A}}^{A} \,\mu_A.
\end{align}
We shall refer to such choice of splitting as \enquote{choosing a gauge}. This is to be distinguished from the projective freedom  $\nabla \in [\nabla]$; that we refer to as a \enquote{choice of scale}. In the following we will repeatedly consider general changes of gauge $L_{\ei{B}} \mapsto \hat{L}_{\ei{B}}$ and parametrise these as
\begin{equation}\label{Change of gauge: definition of chi}
L_{\ei{B}} \qquad  \longmapsto \qquad \hat{L}_{\ei{B}}=L_{\ei{B}}+\chi_A\,\Pi^A_{\,\,\ei{A}}.
\end{equation} 

The linear connection induced by the Cartan connection on $\Tscr$ is summarised in the following Proposition.
\begin{prop}\label{ExtendedConnectionProposition} The linear connection induced on the bundle $\Tscr$ by the Cartan connection is, given a choice of gauge $ L_{\ei{A}}$, described by:
\begin{align}\label{ExtendedConnection} \nabla_c L_{\ei{C}}&=-\upsilon_{Bc}\Pi^B_{\,\,\ei{C}},& \nabla_c I^{\ei{C}}&=0,& \nabla_c \rotpi^{\ei{B}}_{\,\,B}&= \upsilon_{Bc}I^{\ei{B}},& \nabla_c \Pi^B_{\,\,\ei{B}}&=0.  \end{align}
In the above, $\upsilon_{Bc}$ is a cotractor valued $1$-form parametrising the connection and depending on the choice of splitting. Under a change of gauge \eqref{Change of gauge: definition of chi} it transforms according to:
\begin{equation}\label{TransformationRuleCotractorValuedConnectionForm}
\hat{\upsilon}_{Bc}=\upsilon_{Bc} - \nabla_c \chi_B.
\end{equation}
\end{prop}

The curvature tractor is readily computed in terms of the connection coefficient $\upsilon_{cC}$ and the usual projective curvature tractor:
\begin{equation}
\Omega_{ab\phantom{\ei{A}}\ei{B}}^{\phantom{ab}\ei{A}}=2\nabla_{[a}\upsilon_{|A|b]}\, \Pi^A_{\,\,\ei{B}}\,I^{\ei{A}} + \Omega_{ab\phantom{A}B}^{\phantom{ab}A}\,\Pi^B_{\,\,\ei{B}}\,\rotpi^{\ei{A}}_{\,\,A}.
\end{equation}

To summarise this section, we have seen that the non-effective Cartan geometries that we consider, and in particular modelled on \eqref{Homogeneous space model}, are parametrised (see Eqs. \eqref{GeneralParametrisationConnection}, \eqref{ExtendedConnection}) by a projective connection together with a cotractor valued $1$-form $\upsilon_{Bc}$. The latter is related to the non-effectiveness of the model and unconstrained at this stage. In the coming sections we will see that this freedom can be fixed by requiring a holonomy reduction to the Poincaré group $ISO(1,n)\subset G$.

\subsection{Metric holonomy reduction}\label{ssec:metric-holonomy-reduction}

We shall now additionally assume that there is a non-degenerate \emph{parallel} section $H^{\ei{A}\ei{B}}$ of signature $(n+1,2)$,
\begin{equation}\label{Derivative of H}
	\nabla_a H^{\ei{A}\ei{B}} =0,
\end{equation}
 such that the canonical section  $I^{\ei{A}}$ is null; \[\Phi_{\ei{A}\ei{B}}I^{\ei{A}}I^{\ei{B}}=0.\]
 Here $\Phi_{\ei{A}\ei{B}}$ denotes the pointwise inverse of $H^{\ei{A}\ei{B}}$.
 
This amounts to having a holonomy reduction to the Poincaré group $ISO(1,n)\subset G$ and, in this section, we will draw some immediate consequences of these equations.
 
   \subsubsection{Decomposition in a gauge}
 
 Making use of the invariant projection $\Pi : \Tscr \to \mathcal{T}$ and the the canonical tractor $I^{\ei{A}}$, one obtains a pair of projective tractors $(H^{AB} ,I_A)$,
 	\begin{align} H^{AB}&= H^{\ei{A}\ei{B}}\Pi_{\,\,\ei{A}}^A \Pi_{\,\,\ei{B}}^B,& I_B&= \Phi_{\ei{A}\ei{B}}I^{\ei{A}}\rotpi^{\ei{B}}_{\,\,B}. \label{ProjectiveScaleTractorFromExtendedTractorBundle}
 	\end{align}
These are gauge invariant quantities; note that the invariance of $I_B$ relies on the fact that $I^{\ei{A}}$ is null. This condition also implies that $H^{AB}$ is degenerate of kernel $I_A$,
\begin{equation}
 H^{AB}I_B =0.
 \end{equation}

 It will be useful to introduce the following notations for \emph{gauge-dependent} decompositions of $H^{\ei{A}\ei{B}}$.  These notations will be used throughout.
 \begin{defi}\label{Definition: sExtendedMetricComponents}
 	Any gauge $L_{\ei{A}}$ uniquely defines a smooth function $f$, a projective tractors $J^B$, and a symmetric bilinear form $\varphi_{BC}$  on $\mathcal{T}$ such that
 	\begin{align}\label{ExtendedMetricComponents}
 		\begin{alignedat}{2}
 			H^{\ei{A}\ei{B}}&=H^{AB}\,\rotpi^{\ei{A}}_{\,\,A}\rotpi^{\ei{B}}_{\,\,B} + 2 J^B\, I^{(\ei{A}}\rotpi^{\ei{B})}_{\,\,B} + f\, I^{\ei{A}}I^{\ei{B}},\\[0.4em]
 			\Phi_{\ei{B}\ei{C}}&=2I_{C} L_{(\ei{B}}\Pi^C_{\,\,\ei{C})}+\varphi_{BC}\, \Pi^B_{\,\,\ei{B}}\Pi^{C}_{\,\,\ei{C}}.
 		\end{alignedat}
 	\end{align}
 	
 	Since $H^{\ei{A}\ei{C}} \Phi_{\ei{C}\ei{B}} = \delta^{\ei{A}}_{\st \ei{B}}$, they satisfy the relations
 	\begin{align} \label{AlmostInverseProperty + LinkfJvPhi}
 		H^{AC}\varphi_{CB} + J^{A}I_{B}&=\delta^A_{\,\,B}, &
 		fI_A&=-J^B\varphi_{AB}, &
 		J^CI_C&=1.
 	\end{align}

 \end{defi}
 
 We can now turn to the consequences of imposing that $H^{\ei{A}\ei{B}}$ is parallel transported. First, it immediately follows from \eqref{Derivative of I and Pi} and \eqref{Derivative of H} that
  \begin{align} \label{Paralel transport of H, consequences1}
  	\nabla_c H^{AB}&= 0, & \nabla_c I_B&=0.
 \end{align}
  In general one has the proposition. 
  \begin{prop}\label{PropertiesExtendedMetricComponents}
  In any gauge $L_{\ei{A}}$, the equations $\nabla_c H^{\ei{A}\ei{B}}=0$ are equivalent to \eqref{Paralel transport of H, consequences1} together with
  	 	\begin{align}
			\nabla_c J^A &= -H^{AB}\upsilon_{Bc},&
  		\nabla_c f &= -2J^B\upsilon_{Bc},&
		\nabla_c \varphi_{AB}&=2\upsilon_{c(A}I_{B)}.
  	\end{align}
  \end{prop}
 \begin{proof}
The desired relations are obtained directly using $\nabla_cH^{\ei{A}\ei{B}}=0 $ and Proposition~\ref{ExtendedConnectionProposition}.
 \end{proof}
 
  By the results in~\cite{Cap:2014ab,Flood:2018aa}, equations \eqref{Paralel transport of H, consequences1} imply that the projective class $[\nabla]$ is the projective class of a projectively compact Ricci-flat Lorentzian metric.  Recall, from the results reviewed in section \ref{ClassicalProjectiveCompactification}, that $(H^{AB}, I_A)$ will then be parametrised in terms of $\zeta^{ab}$, $\pscale$ as in \eqref{ProjectiveTractor: H and I} and \eqref{TractorMetricBoundaryScale}. For the other components of $H^{\ei{A}\ei{B}}$ we have the following proposition.
 
 \begin{prop}\label{Definition: sExtendedMetricComponents in a scale}
 	Under our assumption that $H^{\ei{A}\ei{B}}$ is parallel
	, any gauge $L_{\ei{A}}$ uniquely defines a projective density $\bar \lambda \in\mathcal{E}(2)$ and a projective tractor $\lambda_A \in \mathcal{T}_A(1)$ such that
 	\begin{align}\label{ExtendedMetricComponents in a scale}
  			\bar \lambda &= \varphi_{AB}X^A X^B, &  \lambda_A& = \varphi_{AB}X^B.
  	\end{align}
 	Moreover,
\begin{align}\label{Extended metric expression as a function of lambda}
	\begin{alignedat}{3}
	\varphi_{AB}&= D_B\lambda_A - \pscale\upsilon_{Ab}Z^b_{\st B} - \upsilon_{Cb}Z^b_{\st B} X^C I_A,\\[0.4em]
	\pscale J^A & = X^A - H^{AB} \lambda_B,\\[0.4em]
	\pscale^{2}f& = - \bar \lambda  + H^{AB} \lambda_A\lambda_B.
\end{alignedat}
\end{align}
 \end{prop}
In other words, in a gauge $L_{\ei{A}}$, the triplet ($\nabla, H^{\ei{A}\ei{B}}, I^{\ei{a}}$) is entirely parametrised by ($\zeta^{ab}$, $\pscale$) and ($\upsilon_{Cb}$, $\lambda_A$). As we shall see, these last two quantities are essentially pure gauge and everything can in fact be determined by the metric data ($\zeta^{ab}$, $\pscale$).

 \begin{proof}\mbox{}
 	
 	By definition, $\lambda_A = \varphi_{AB}X^B$, hence taking the Thomas derivative:
 	\[ D_C\lambda_A = (D_C\varphi_{AB})X^B + \varphi_{AC}=\nabla_c\varphi_{AB}Z^c_{\st C}X^B + \varphi_{AC}. \]
 	Using Proposition~\ref{PropertiesExtendedMetricComponents} this becomes
 	\begin{align*}
 	 D_C\lambda_A = 2\upsilon_{c(A}I_{B)} Z^c_C X^B +  \varphi_{AC} =  \pscale \upsilon_{cA}Z^c_C + \upsilon_{cB}X^BI_AZ^c_C +  \varphi_{AC},
 	\end{align*}
 	which gives the desired expression for $ \varphi_{AB}$.
 	
 	Second, contracting the equation on the left of \eqref{AlmostInverseProperty + LinkfJvPhi} with $X^B$ gives 
 	\begin{equation*}
 		H^{AC}\varphi_{CB} X^B + J^A \pscale = X^A
 	\end{equation*}
 	and making use of the definition $\lambda_A = \varphi_{AB}X^B$ 
	 this yields the second relation in \eqref{Extended metric expression as a function of lambda}. In order to obtain the third, we contract the second with $\lambda_A$
 	\begin{equation*}
 		\pscale J^A \lambda_A  = \bar \lambda - H^{AB}\lambda_A\lambda_B
 	\end{equation*}
 	and rewrite the left-hand side as
 	\begin{equation*}
 	 \pscale J^A \lambda_A =  \pscale J^A \varphi_{AB} X^B = -\pscale f I_A X^B = - \pscale^2 f,
 	\end{equation*} 	
 	where we made use of the second equation in \eqref{AlmostInverseProperty + LinkfJvPhi}.
 \end{proof}
 
 \paragraph{Orbit decomposition}\mbox{}
 
 It follows from \eqref{Paralel transport of H, consequences1}, and the results from~\cite{Cap:2014ab,Flood:2018aa}
 that the projective class $[\nabla]$ is the projective class of a Ricci-flat projectively compact of order 1 Einstein metric:
 \begin{equation}
 \pg_{ab} = \pscale^{-2} \zeta_{ab}
 \end{equation}
 with $\zeta^{ab}=	H^{AB}Z_A^aZ_B^b$, $\pscale = X^AI_A$. We have already summarised in Section \ref{ClassicalProjectiveCompactification} some consequences of this fact.
 
 Note that whilst $\pscale \in \mathcal{E}(1)$ is a gauge-independent quantity, and gives the curved orbit decomposition
 \begin{equation*}
 	\Q = \Qint \cup \B,
 \end{equation*}
 $\bar\lambda \in \mathcal{E}(2)$, defined in Eq.~\eqref{ExtendedMetricComponents in a scale}, is a gauge-\emph{dependent} quantity: under the change of gauge \eqref{Change of gauge: definition of chi} one has
 \begin{equation}\label{GaugeChangeLambdaBar} \bar \lambda \mapsto \bar\lambda -2\pscale\chi_AX^A.\end{equation}
 However, its restriction to $\B = \mathcal{Z}(\pscale)$,
  \begin{align}
 \dscale := i^* \bar \lambda \; \;\in \mathcal{E}_{\B}(2),
 \end{align}
  \emph{is} gauge independent. Consequently, the decomposition of $\B$ according to the sign of $\dscale$
  \begin{equation}
  	\B = \mathcal{H}^- \cup  \mathcal{H}^0  \cup \mathcal{H}^+
  \end{equation}
    is geometrically relevant. We will see that $\dscale$ in fact coincides with the distinguished scale on $\B$ discussed in Section~\ref{ClassicalProjectiveCompactification} (see the discussion before Theorem \ref{Projective compactification review: CapGoverBoundaryDefiningFunction}).

  \subsubsection{The extended connection on the interior $\Qint$}
  
   Our first result is to prove that, on $\Qint$, the extended connection is in fact completely fixed by the holonomy reduction \eqref{Derivative of H}.

\begin{theo}\label{Theorem: UniquenessOnInside}\mbox{}
	
On $\Qint$, both the connection $\nabla_c$ on $\Tscr$ and $H^{\ei{A}\ei{B}}$ are completely determined by $\zeta_{ab}$ and $\pscale$. More precisely:
\begin{enumerate}
\item On $\Qint$ there is a unique gauge $\mL_{\ei{A}}$ such that
\begin{equation}\label{Definition: Metric gauge}
	 \hat{\lambda}_A := \mvphi_{AB}X^B=0.
\end{equation}
\item In this preferred gauge, the connection components \eqref{ExtendedConnection} are given by
\begin{equation}
\pgctform_{Ac}=-\pscale^{-1}\zeta_{cb}Z^b_{\st A},
\end{equation}
and the metric components \eqref{ExtendedMetricComponents} by
\begin{align}
\mf&=0, & \mJ^B&=\pscale^{-1}X^B, & \mvphi_{BC}&=\zeta_{bc}Z^b_{\st B}Z^c_{\st C}.
\end{align}
\end{enumerate}
\end{theo}
We shall refer to this specific choice of splitting on $\Qint$ as the \emph{metric gauge}; the components of tractors in this gauge will be written with hatted letters. Note that, even though at this stage it is natural to also fix the scale $\nabla$ to be the Levi-Civita connection of $\sigma^{-2}\zeta_{ab}$, one is not forced in anyway to do that: the above condition fixes the gauge $L_{\ei{a}}$ but leaves a complete freedom in the scale $\nabla$.

\begin{proof}[Proof of Theorem~\ref{Theorem: UniquenessOnInside}] \phantom{a}
\begin{enumerate}
\item \textit{Existence and uniqueness of metric gauge.} \smallskip\newline
Under a generic change of gauge \eqref{Change of gauge: definition of chi}, $\varphi_{AB}$ transforms according to
\[\varphi_{AB} \;\mapsto \;\hat{\varphi}_{AB}= \varphi_{AB} - 2 \chi_{(A}I_{B)},\] contracting with $X^B$, one sees that we are trying to achieve
\[\varphi_{AB} \;\mapsto\; 0= \varphi_{AB}X^B - \chi_A \pscale - \chi_BX^BI_A. \]
On $\Qint$ this is uniquely solved by \[\chi_A=\varphi_{AB}X^B\pscale^{-1} -  \frac{1}{2}\varphi_{BC}X^BX^C\pscale^{-2}I_A=\lambda_A\pscale^{-1}-\frac{1}{2}\bar\lambda\pscale^{-2}I_A.\]

\item \textit{Computation of components of $H^{\ei{A}\ei{B}}$ and the connection coefficient $\pgctform_{Ab}$ } \smallskip\newline

Taking $\hat{\lambda}_A =0$ in \eqref{Extended metric expression as a function of lambda} yields $\hat f =0$, $\hat J^A  = \pscale^{-1} X^A$ and 
\begin{equation}
\hat \varphi_{AB} = - \pscale \hat v_{Ab} Z^b{}_B - \hat v_{Cb} Z^b{}_B X^C I_A.
\end{equation}
Since $0= \hat \lambda_A = 	\hat \varphi_{AB} X^B$, we must have $ \mvphi_{BC}=\xi_{bc}Z^b_{\st B}Z^c_{\st C}$ for some $\xi_{bc}$ and the above equation is uniquely solved as
\begin{align*}
	  v_{Ab}= -\pscale^{-1}\xi_{bc} Z^c_{\st A}.
\end{align*}
Finally making use of  $ \mvphi_{BC}=\xi_{bc}Z^b_{\st B}Z^c_{\st C}$, \eqref{TractorMetricBoundaryScale} and \eqref{AlmostInverseProperty + LinkfJvPhi} one finds
\begin{align*}
	\delta^A_{\st B} &=  H^{AC} \xi_{bc}Z^b_{\st B}Z^c_{\st C} + \pscale^{-1} X^A D_B{\pscale}\\
	& = \zeta^{ac} \xi_{bc} \; Z^b_{\st B} W_a^{\st A} - \pscale^{-1} \zeta^{ac}\xi_{bc} \nabla_a \pscale  Z^b_{\st B} X^A  + \pscale^{-1} X^A D_B{\pscale}
\end{align*}
which is equivalent to $\zeta^{ac} \xi_{bc} = \delta^a_{\st b}$.

\item The expressions that we derived for the connection form and the components of $H^{\ei{A}\ei{B}}$ only involve $\pscale$ and $\zeta_{ab}$.  Since $H^{AB}$ is also completely determined by $\zeta$, we conclude that the structure is completely determined by $\pscale$ and $\zeta_{ab}$ on the $\Qint$.
\end{enumerate}
\end{proof}

Clearly, the formulae in Theorem \ref{Theorem: UniquenessOnInside} do not extend to the boundary and hence do not teach us anything about how such a structure should look there. This will be important when studying the question of existence of these structures in Section~\ref{ssec:existence}. In the next section, we shall devise a natural way to fix the gauge at $\mathcal{H}^\pm$, however, we will observe that it implicitly requires an additional non-canonical choice, that we will later understand to be a section of the extended boundary $\eB$.

\subsection{Projective compactification}\label{ssec:AbstractDiscussionProjectiveCompactification}
The aim of this section is to construct a distinguished gauge at boundary points; throughout this section we will only be interested in boundary points \emph{away} from null infinity $\mathcal{H}^0 = \B \cap  \mathcal{Z}(\bar \lambda)$, where conformal geometry and conformal compactness provide the relevant tools, and concentrate on $\mathcal{H}^\pm$ where $\bar \lambda\neq 0$. 

In Theorem~\ref{Theorem: UniquenessOnInside}, it was established that there was a uniquely determined gauge on $\Qint$ such that the gauge dependent fields, 
\begin{align*}
\lambda_A&:=\varphi_{AB}X^B, & \lambda&:=\varphi_{AB}X^B X^A,
\end{align*}
vanish entirely. This gauge is however singular at points where $\hat{\sigma}=0$: this is because $\lambda$ becomes gauge invariant along the boundary, see Eq. \eqref{GaugeChangeLambdaBar}, and by definition cannot vanish on $\mathcal{H}^{\pm}$.  We are going to use $\lambda_A$ and $\bar\lambda$ to define alternative gauges, better suited than the metric gauge near boundary points of $\mathcal{H}^\pm$.

Our first observation is:
\begin{lemm}\label{Lemma: BoundaryGaugeStep1}
For any everywhere regular gauge $L_\mathcal{A}$ there is, at points where $\lambda \neq0 $, a \underline{unique} projective scale $\bnabla_a \in [\bnabla_a]$ such that
\begin{equation}\label{Definition: Boundary gauge step 1}
	\lambda_A  = \bar \lambda\; Y_A.
\end{equation}
In order to preserve this condition under a change of gauge $L_\ei{A} \mapsto L_\ei{A}+\chi_A\Pi^A_{\st\ei{A}}$, then one must perform simultaneously a change of projective scale $\bnabla_a \mapsto \bnabla_a +\Upsilon_a$ such that
\begin{equation}\label{MaintainingTheCondition}
	\Upsilon_a = \frac{\chi_0\nabla_a \pscale + \pscale\chi_a}{ \bar\lambda -2\chi_0 \pscale } \end{equation}
where $\chi_A = \chi_0Y_A+\chi_aZ^a_{\st A}$ in the original projective scale.
\end{lemm}
\begin{proof}
Let at first $\bnabla_a$ be an arbitrary projective scale. We shall show that there exists a unique $1$-form $\Upsilon_a$ such that $\nabla_a +\Upsilon_a$ satisfies the required condition.
Using Equation~\eqref{ChangeOfProjectiveSplitting}, $\Upsilon_a$ must solve:
\[ \lambda_A W^{\st A}_{a} \quad \mapsto \quad 0= \lambda_AW^{\st A}_{a} + \Upsilon_a\, \varphi_{AB}X^BX^A \]
which has a unique solution 
\begin{equation}\label{Proof, boundary gauge: change of scale}
\Upsilon_a = -\bar\lambda^{-1}\lambda_AW^A_a.
\end{equation} 

In order to prove~\eqref{MaintainingTheCondition}, we need to study how $\lambda_A W^A_a$ changes under a simultaneous change of gauge $L_\ei{A} \mapsto L_\ei{A}+\chi_A\Pi^A_{\st\ei{A}}$ and projective connection $\nabla \mapsto \nabla + \Upsilon$. 

Under a change of gauge $\varphi_{AB}  \mapsto \varphi_{AB} - 2\chi_{(A}I_{B)}$ and thus
\begin{align*}
	\bar \lambda &\mapsto \bar \lambda - 2\chi_0 \pscale, & 0= \lambda_A W^{\st A}_{a} \mapsto -\chi_0 \nabla_a \pscale -\pscale \chi_a. 
\end{align*}
It then follows from \eqref{Proof, boundary gauge: change of scale} that we now need to make the change of scale
 \[ \Upsilon_a = -\big( \bar \lambda - 2\chi_0 \pscale \big)^{-1} \; \big(  -\chi_0 \nabla_a \pscale - \pscale \chi_a\big)\]
to preserve the condition $\lambda_A W^{\st A}_{a} =0$.
\end{proof}

\begin{rema}
Observe that if a gauge is regular at points of $\mathcal{H}^0 :=\B\cap \mathcal{Z}(\bar\lambda)$, then the scale $\bnabla$ given by the proposition cannot extend there. Indeed, if it could, then the condition $\lambda_A=\bar\lambda Y_A$ would also extend by continuity. Plugging this into the second equation of~\eqref{Extended metric expression as a function of lambda}, we would have $\pscale J^A \to X^A$ at these points, but this can only happen if $J^A \sim \pscale^{-1}X^A$, contradicting the regularity of $H^{\ei{A}\ei{B}}$. More generally, if any gauge and scale $(L_\mathcal{A},\nabla_a)$ are such that together $\lambda_A=\bar\lambda Y_A$ holds then at least one of them must be singular near points of $\mathcal{H}^0$. 
\end{rema}

We recall that, in any scale $\bnabla$, $H^{AB}$ must be of the form \eqref{TractorMetricBoundaryScale}. The choice of scale \eqref{Definition: Boundary gauge step 1} will allow to have similar expressions for the extended metric.
\begin{prop}\label{Proposition: Extended metric expression in gauge1}
In a choice of gauge and scale $(L_{\ei{a}}, \nabla)$ satisfying  \eqref{Definition: Boundary gauge step 1}, and therefore excluding points of $\mathcal{H}^0$, the components \eqref{ExtendedMetricComponents} of the extended metric $H^{\ei{a}\ei{b}}$ read
 \begin{align}\label{Extended metric expression as a function of lambda, in boudary scale1}
	\begin{alignedat}{3}
		\varphi_{AC}&= \left(\bar\lambda P_{ca} -\frac{1}{2}\pscale^{-1}\nabla_{c}\bar\lambda \nabla_{a} \pscale -\pscale \upsilon_{Bc}W^{B}_a \right)\; Z^a_{\st A}Z^c_{\st C} +\bar\lambda \;  Y_AY_B,\\[0.75em]
		 J^A & = \left( \frac{\bar \lambda}{\pscale^2} \zeta^{ab}\bnabla_b\pscale \right)  \;W^{A}_{a}+ \pscale^{-1}\left(1 - \frac{\bar \lambda}{\pscale^2}\zeta^{ab}\bnabla_a\pscale\bnabla_b\pscale \right) \;X^A ,\\[0.75em]
		f& = - \frac{\bar \lambda}{\pscale^2}\left(  1   -   \frac{\bar \lambda}{\pscale^2} \; \zeta^{ab} \nabla_a \pscale \nabla_b \pscale \right).
	\end{alignedat}
\end{align}
In particular, since the extended metric is finite along the boundary, we recover (see the discussion around \eqref{Projective compactification review: N and nu def}), that the vector field $N^a$ and density $\nu$
\begin{align*}
	N^a &=  \pscale^{-2} \zeta^{ab}\bnabla_b\pscale, & \nu&= N^a \bnabla_a \pscale,
\end{align*}
extend smoothly to points of $\B$. On top of this, we obtain that the restriction of $\nu^{-1}$ and $\bar \lambda$ must coincide at the boundary to give the preferred boundary density
\begin{equation}
 \lambda_0 = \iota^* \nu^{-1} = \iota^* \bar \lambda.
\end{equation}
Finally, the component $\upsilon_{Ab}$ of the connection must satisfy
\begin{equation}
	\upsilon_{Cb} X^C = \pscale^{-1}\,\frac{1}{2} \,\nabla_b \bar \lambda.
\end{equation}
and since it must be finite this means that our scales $\nabla$ satisfy $\iota^* \nabla_c \bar \lambda =0$. In particular, by Theorem \ref{Projective compactification review: CapGoverBoundaryDefiningFunction}, along the boundary, it is a special scale defined by the density $ \sqrt{\lvert \lambda_0\rvert} = \lvert\detd \bar h\rvert^\frac{1}{2}$.
\end{prop}

\begin{proof}
Taking \eqref{Extended metric expression as a function of lambda} from Proposition \ref{Definition: sExtendedMetricComponents in a scale} gives, making use of our choice of scales satisfying \eqref{Definition: Boundary gauge step 1},
\begin{align*}
		\varphi_{AB}&=   D_B ( Y_A \bar \lambda)- \pscale\upsilon_{Ab}Z^b{}_B - \upsilon_{Cb}Z^b{}_B X^C I_A,\\[0.4em]
		\pscale J^A & = X^A - \bar \lambda \,H^{AB} Y_B,\\[0.4em]
		\pscale^{2}f& = - \bar \lambda    + \bar \lambda^2 \, H^{AB} Y_AY_B.
\end{align*}
Making use of \eqref{TractorMetricBoundaryScale} one then obtains the desired expressions for $J^A$ and $f$. Appealing to the definitions \eqref{Definition: Thomas operator} and \eqref{TractorConnection} for the Thomas $D$-operator and the tractor connection, the expression of $\varphi_{AB}$ can be rewritten as
\begin{align*}
	\varphi_{AB} 	&=\bar \lambda\;  Y_A Y_B  + \nabla_b \bar \lambda \;Y_A Z^b_{\st B}  +  \bar \lambda \, P_{ab}Z^a_{\st A} Z^b_{\st B} - \pscale\upsilon_{Ab}Z^b_{\st B} - \upsilon_{Cb}Z^b_{\st B} X^C I_A
\end{align*}
since, in our scale, $\varphi_{AB} X^B = \bar \lambda Y_A$, this implies that 
\begin{equation*}
	\upsilon_{Cb} X^C = \pscale^{-1}\,\frac{1}{2} \,\nabla_b \bar \lambda
\end{equation*}
and therefore
\begin{align*}
	\varphi_{AB} W^A_{a}  W^B_{b}	&=\bar \lambda \, P_{ab} - \pscale\upsilon_{Ab}W^A_{a}- \pscale^{-1} \frac{1}{2} \nabla_b \bar \lambda \nabla_b \pscale.
\end{align*}
\end{proof}

This Proposition means that the extended metric is entirely parametrised by $\zeta^{ab}$, $\pscale$, $\bar \lambda$ and the coefficient $\upsilon_{Ab}$ of the connection. We will now impose a condition on the gauge, which will leave us with a very small amount of freedom:
\begin{prop}\label{Prop: BoundaryGaugeStep2}
	Let $x_0$ be a point which does not belong to $\mathcal{H}^0$. Then there always exists a neighbourhood of $x_0$ and local choice of gauge $L_\ei{A}$ such that
	\begin{equation}\label{Definition: Boundary gauge step 2}
	f := H^{\ei{a}\ei{b}} L_{\ei{a}} L_{\ei{b}} =0.
\end{equation}
\end{prop}
\begin{proof}\mbox{}
	
	If $x_0$ is in the interior $\Mint$, one can always pick the metric gauge (see Theorem \ref{Theorem: UniquenessOnInside}).
	
	If $x_0$ is in $\mathcal{H}^{\pm}$ then, since $\bar\lambda\neq 0$ at $x_0$, we may assume, restricting if necessary to a smaller neighbourhood, that $\bar\lambda \neq 0$.	
The expression for the gauge shows that what we are trying to achieve only depends on a choice of gauge and does not imply anything on the choice of scale: it is therefore consistent to require \eqref{Definition: Boundary gauge step 1}.  We now consider the change of gauge
	 \begin{align*}
	 	H^{\ei{a}\ei{b}} L_{\ei{a}} L_{\ei{b}}\qquad  \mapsto\quad 0&=H^{\ei{a}\ei{b}} \tilde L_{\ei{a}} \tilde L_{\ei{b}}\\
	 	&= H^{\ei{a}\ei{b}}(L_{\ei{a}}+\chi_A\Pi^A_\ei{a})(L_{\ei{b}}+\chi_B\Pi^B_\ei{b})\\
	 	&=f + H^{AB} \chi_A \chi_B +  2 J^B \chi_B.
	 \end{align*}
	Making use of \eqref{Extended metric expression as a function of lambda, in boudary scale1} and the expression \eqref{Projective compactification review: N and nu def} for $N$ and $\nu$, this can be rewritten as
		\begin{align}\label{Proof of the f=0 gauge. Achieving the condition}
		0	&= f + \chi_0^2 \nu + \zeta^{ab}\chi_a \chi_b - 2 \pscale \chi_0 N^a \chi_a +2\bar \lambda N^a \chi_a + 2\pscale^{-1}  \chi_0 (1 - \bar \lambda \nu ) .
	\end{align}
	Since $\bar\lambda\neq 0$ note that $\pscale^{-1}(1-\bar\lambda\nu)=-\pscale f \bar\lambda^{-1}$, so:
	\begin{align*}
		0	&= f + \chi_0^2 \nu + \zeta^{ab}\chi_a \chi_b - 2 \pscale \chi_0 N^a \chi_a +2\bar \lambda N^a \chi_a - 2\bar\lambda^{-1}\pscale \chi_0 f  .
	\end{align*}
	 We only need to prove that a solution to this equation exists; let us take as an ansatz $\chi_0=0$ and $\chi_a = \chi \nabla_{a}\pscale$ we find:
	 		\begin{align*}
	 	0	&= f  + \chi^2 \zeta^{ab} \nabla_{a}\pscale \nabla_{b}\pscale  +2\chi N^a \nabla_{a}\pscale   \bar \lambda\\
	 	&= \chi^2\; \pscale^{2}  \nu +\chi\; 2 \bar \lambda \nu +f.
	 \end{align*}
	 	 Now the (reduced) discriminant of this quadratic equation is $\bar\lambda^2\nu^2 - \pscale^2 f\nu$: it must be positive in a neighbourhood of the boundary $\pscale =0$ and therefore one can always achieve the gauge with a term of the form
	 \begin{align}
	 	\chi &=  -\frac{\bar \lambda}{\pscale^2}\left( 1- \sqrt{1-f\bar\lambda^{-2}\pscale^2\nu^{-1}} \right) =-\frac{1}{2} f\, \bar \lambda ^{-1}\nu^{-1} + O(\pscale).
	 \end{align}
\end{proof}

\begin{rema}
We point out that there is a slight subtlety at this point, it is not clear from the above that there is a gauge and \emph{special} projective scale in which we have these properties. In other words, there may be no density $\tau$ for which $\bnabla_a \tau=0$. \end{rema}

We shall now require both \eqref{Definition: Boundary gauge step 1} and \eqref{Definition: Boundary gauge step 2} and draw the consequences of these choices. We recall that $N^a$, $\nu$ 
 are given by the expressions \eqref{Projective compactification review: N and nu def}. 
\begin{prop}\label{Proposition: Extended metric expression in gauge2}
	Let $x_0$ be a point away from $\mathcal{H}^0$. Let $(L_{\ei{a}}, \nabla)$ be a choice of gauge and scale in a neighbourhood of $x_0$ satisfying both \eqref{Definition: Boundary gauge step 1} and \eqref{Definition: Boundary gauge step 2}. Then, the components \eqref{ExtendedMetricComponents} of the extended metric $H^{\ei{a}\ei{b}}$ satisfy
		\begin{align}\label{Extended metric expression, in boudary scale 2 a)}
			f& =0, & J^A &= \bar  \lambda N^a\, W_a^A + \pscale^{-1}(1-\bar\lambda \nu)X^A, & \bar \lambda (1- \bar \lambda \nu) &= 0,
	\end{align}
	\begin{align}\label{Extended metric expression as a function of lambda, in boudary scale 2 b)}
			\varphi_{AB}&= \left( \zeta_{ab} - \bar \lambda \pscale^{-2} \bnabla_a \pscale \bnabla_b \pscale\right)\; Z^a_{\st A}Z^b_{\st B} +\bar \lambda\;  Y_AY_B.
	\end{align}
	In particular, $q_{ab}:=\zeta_{ab} - \bar \lambda \pscale^{-2} \bnabla_a \pscale \bnabla_b \pscale$ extends smoothly to points of $\B$ in this neighbourhood. What is more, it also follows from \eqref{Definition: Boundary gauge step 1} and \eqref{Definition: Boundary gauge step 2} that, when $\bar\lambda \neq 0$, $N^a q_{ab}=0$ . 
	Finally the components $\upsilon_{Ab}$ of the connection are given by
	\begin{align}\label{Extended connection expression in boudary scale 2}\begin{aligned}\upsilon_{Ab} X^A &= \pscale^{-1}\,\frac{1}{2} \,\nabla_b \bar \lambda, \\[0.7em]  \upsilon_{Ab} W^A{}_a &= -\pscale^{-1} \Big( q_{ab} -\bar \lambda P_{ab}  + \frac{1}{2}\pscale^{-1}\nabla_a \bar \lambda \nabla_b\pscale  \Big).
		\end{aligned}
	\end{align}
In particular, since the extended connection is finite along the boundary, we recover, at points where $\bar \lambda \neq0$, the identities \eqref{Projective compactification review: identities on schouten} on the projective Schouten tensor.
\end{prop}
\begin{proof}\mbox{}
	
	Equations \eqref{Extended metric expression, in boudary scale 2 a)} are directly obtained from \eqref{Extended metric expression as a function of lambda, in boudary scale1} under the assumption that $f=0$. It also implies \begin{equation*}
	\varphi_{AB}= q_{ab}\; Z^a_{\st A}Z^b_{\st B} +\bar \lambda \;  Y_AY_B,
	\end{equation*}
	where $q_{ab} = \left(\bar\lambda P_{ba} -\frac{1}{2}\pscale^{-1}\nabla_{b}\bar\lambda \nabla_{a} \pscale -\pscale \upsilon_{Cb}W^{C}_a \right)$. Then the first equation in \eqref{AlmostInverseProperty + LinkfJvPhi} gives, when contracted with $W^B{}_{b}Z^a_{\st A}$,
	\[ \zeta^{ac} q_{bc} + \bar\lambda N^a\bnabla_b\pscale = \delta^{a}_{\st b},\] 
	or equivalently at points where $\pscale\neq 0$
	\[ q_{ab} + \bar\lambda \pscale^{-2} \bnabla_a\pscale\bnabla_b\pscale= \zeta_{ab}.\]
 Therefore one has
	\begin{equation*}
	q_{ab} = \zeta_{ab} - \bar\lambda \pscale^{-2}\bnabla_a\pscale\bnabla_b\pscale= \left(\bar\lambda P_{ba} -\frac{1}{2}\pscale^{-1}\nabla_{b}\bar\lambda \nabla_{a} \pscale -\pscale \upsilon_{Cb}W^{C}{}_a \right)
	\end{equation*}
	which gives the required expression for $\upsilon_{Cb}W^{C}{}_a$.

	Finally, it follows immediately from $f=0$ and the second equation in \eqref{AlmostInverseProperty + LinkfJvPhi} that
	\begin{equation*}
		0=J^A\varphi_{AB}W^B{}_b = \bar \lambda N^a q_{ab}.
	\end{equation*}
	 
\end{proof}

\begin{rema}
	If the gauge $L_{\ei{a}}$ is such that $\bar\lambda=0$, which is however impossible on $\mathcal{H}^{\pm}$, then the gauge is in fact the metric gauge and the expression in Proposition \ref{Proposition: Extended metric expression in gauge2} coincide with those of Theorem~\ref{Theorem: UniquenessOnInside}.
\end{rema}

The message from  Proposition \ref{Proposition: Extended metric expression in gauge2} is that under the joint conditions  \eqref{Definition: Boundary gauge step 1}, \eqref{Definition: Boundary gauge step 2} for $(L_{\ei{A}},\nabla)$, which we reproduce here for convenience,
\begin{align}\label{Gauge fixing conditions}
	f&=0,& \lambda_A &= \bar \lambda Y_A,
\end{align}
 both the extended metric and the extended connection are parametrised by $\zeta^{ab}$ and $\pscale$. The combined conditions \eqref{Gauge fixing conditions} on the gauge and scale can always be achieved locally, in a neighbourhood away from $\mathcal{H}^0$, and it follows from Lemma \ref{Lemma: BoundaryGaugeStep1} and the details of the proof of Proposition \ref{Prop: BoundaryGaugeStep2} (see eq. \eqref{Proof of the f=0 gauge. Achieving the condition}) that, if $L_\ei{A} \mapsto L_\ei{A}+\chi_A\Pi^A_{\st\ei{A}}$, $\bnabla_a \mapsto \bnabla_a +\Upsilon_a$ describes a transformation from $(L_\ei{A},\bnabla_a)$ to another couple satisfying~\eqref{Gauge fixing conditions} then the parameters $(\chi_A, \Upsilon_a)$ must satisfy
\begin{align}\label{Change of gauge conditions}
		\Upsilon_a &= \frac{\chi_0\nabla_a \pscale + \pscale\chi_a}{ \bar\lambda -2\chi_0 \pscale },\\[0.4em]
			0	&=\chi_0^2 \nu  - 2\chi_0\big( \pscale  N^a \chi_a  -\pscale^{-1} (1 - \bar \lambda \nu )\big) + \zeta^{ab}\chi_a \chi_b +2\bar \lambda N^a \chi_a.
\end{align}
 The second equation can be solved for $\chi_0$ and substituted into the first to give a meaningful value for $\Upsilon$. In this sense, one should think of $\chi_a$ as being a free parameter determining the others. 
 
It follows from Lemma \ref{Lemma: BoundaryGaugeStep1} that a gauge $L_{\ei{a}}$ such that $f=0$, existence being guaranteed by Proposition \ref{Prop: BoundaryGaugeStep2}, uniquely defines, at point where $\lambda \neq0$, a pair $(L_{\ei{a}}, \nabla)$ satisfying our gauge fixing condition \eqref{Gauge fixing conditions} (if $\lambda =0$, then the gauge is the metric gauge and any scale $\nabla$ will then satisfy \eqref{Gauge fixing conditions}). We can now investigate the converse.

\begin{lemm}
Let $\nabla$ be some scale and suppose that both $(L_{\ei{a}}, \nabla_a)$ and $(L_\ei{A}+\chi_A\Pi^A_{\st\ei{A}}, \bnabla_a)$ satisfy \eqref{Gauge fixing conditions}. Then we must have $\chi_a = -\chi_0 \, \pscale^{-1}\nabla_a \pscale$ and there are only two possibilities for $\chi_0$, depending on whether $\bar \lambda$ was zero for $L_{\ei{a}}$; i.e. depending on whether $L_{\ei{a}}$ was in fact the metric gauge:
 \begin{align}
\text{if} \quad  \bar \lambda &\neq0& \text{then}&& \chi_0 &=\frac{\nu^{-1}}{2\pscale}&& \text{and}\;& \bar \lambda&=\nu^{-1} \;\mapsto\; 0,\\
\text{if} \quad  \bar \lambda &=0& \text{then}&& \chi_0 &= -\frac{\nu^{-1} }{2 \pscale }&& \text{and}\;& \bar \lambda&=0 \;\mapsto\; \nu^{-1}.
\end{align}
\end{lemm}
In other terms, for a given scale $\nabla$, and away from $\mathcal{H}^0$, there is $\mathbb{Z}_2$ ambiguity on the space of gauges satisfying \eqref{Gauge fixing conditions}: This $\mathbb{Z}_2$ action takes us back and forth between the metric gauge, $\lambda_A =0$, and another gauge uniquely associated to $\nabla$. (We recall that the metric gauge does not fix any scale at all and thus there is no contradiction).
\begin{proof}
Let $\nabla$ be some scale and suppose that $(L_{\ei{a}}, \nabla)$ satisfies \eqref{Gauge fixing conditions}. If $(L_\ei{A}+\chi_A\Pi^A_{\st\ei{A}}, \bnabla_a)$ also satisfies  \eqref{Gauge fixing conditions} then, from \eqref{Change of gauge conditions}, one should have:
\begin{align*}
	0 &= \chi_0\nabla_a \pscale + \pscale\chi_a,&
0	&=\chi_0^2 \nu  - 2\chi_0\big( \pscale  N^a \chi_a  -\pscale^{-1} (1 - \bar \lambda \nu )\big) + \zeta^{ab}\chi_a \chi_b +2\bar \lambda N^a \chi_a.
\end{align*}
Solving the first equation gives $\chi_a = -\chi_0 \, \pscale^{-1}\nabla_a \pscale$ on $\Mint$ and injecting this in the second yields
\begin{align}\label{conditions for changing gauge while preserving the scale}
	0	&= 2\chi_0\left(2\chi_0\nu  + \pscale^{-1}(1-2\bar\lambda \nu)\right)
\end{align}
Therefore, there is only one non-trivial admissible choice which is $\chi_0 = \pscale^{-1}\nu^{-1}(\bar\lambda\nu - \frac{1}{2})$. Now since, under a general change of gauge,
\begin{equation}
	\lambda_A = \bar \lambda Y_A \quad\mapsto\quad (\bar \lambda -2\pscale \chi_0) \,Y_A -(\chi_0 \nabla_a \pscale + \pscale \chi_a)\, Z^a_{\st A},
\end{equation}
and given the constraint $0=\bar\lambda(1 - \bar \lambda \nu)$ (see \eqref{Extended metric expression, in boudary scale 2 a)}), there is a unique (non-trivial) possibility, which however depends on whether or not $\bar \lambda$ was zero in the first place :
 \begin{align}
 	\text{if} \quad  \bar \lambda &\neq0& \text{then}&& \chi_0 &= \frac{\bar \lambda}{2 \pscale}=\frac{\nu^{-1}}{2\pscale}&& \text{and}\;& \bar \lambda&=\nu^{-1} \;\mapsto\; 0,\\
 	\text{if} \quad  \bar \lambda &=0& \text{then}&& \chi_0 &= -\frac{\nu^{-1} }{2 \pscale }&& \text{and}\;& \bar \lambda&=0 \;\mapsto\; \nu^{-1}.
 	\end{align}
\end{proof}

Now we remind the reader that the metric gauge $\bar \lambda=0$ does not extends smoothly to the boundary $\B$ and therefore requiring a smooth extension will in fact uniquely fix the gauge. 

\begin{prop}\label{Proposition: existence and unicity of the gauge}
Let $x_0$ be a point in $\mathcal{H}^{\pm}$. Let $\nabla$ be any scale which extends smoothly near $x_0$ such that that $N^a$, as defined by \eqref{Projective compactification review: N and nu def}, also has a smooth extension with non-vanishing boundary value. 
 
Then there exists a unique gauge $L_{\ei{a}}$ satisfying \eqref{Gauge fixing conditions} and extending smoothly to a neighbourhood of $x_0$. It is obtained from the metric gauge $\hat{L}_{\ei{a}} \mapsto L_{\ei{a}} = \hat{L}_{\ei{a}} + \chi_A \Pi^A{}_{\ei{a}}$ by taking
	\begin{align}\label{RequiredChangeFromMetricGauge}
	\chi_A &= -\frac{1}{2}\nu^{-1}\pscale^{-1}\Big(\, Y_A - \pscale^{-1}\bnabla_a\pscale \, Z^a_A \,\Big).
\end{align}
\end{prop}
\begin{proof}\mbox{}
		
Let us start by proving unicity. First consider the metric gauge $\hat{L}_{\ei{a}}$. The pair $(\nabla, \hat{L}_{\ei{a}})$ satisfies \eqref{Gauge fixing conditions}; however, since we identically have $\bar \lambda=0$ in this gauge, this is not an admissible gauge along the boundary. We therefore look for another pair $(\nabla, L_{\ei{a}})$ of gauge and scale, with $L_{\ei{a}} = \hat{L}_{\ei{a}} + \chi_A \Pi^A{}_{\ei{a}}$,  satisfying the gauge fixing conditions \eqref{Gauge fixing conditions}. In other terms, we are looking for a solution to \eqref{Change of gauge conditions} where $\Upsilon_a=0$. As we already saw, under the assumption that $\bar \lambda=0$, this has a unique solution given by
\begin{align*}
	\chi_a & = -\chi_0\,\pscale^{-1}\nabla_a\pscale, & \chi_0 & =-\frac{1}{2} \pscale^{-1} \nu^{-1}.
\end{align*}

We shall now prove that given $\bnabla$ such a gauge always exists in a neighbourhood of $x_0$, we emphasise that it should be understood that we work locally. Let $\tilde{L}_{\ei{a}}$ be some regular gauge in a neighbourhood of $x_0$. Since it is regular at $x_0$, we can assume, restricting the neighbourhood if necessary, $\bar \lambda \neq 0$. Now, by assumption, $\nu^{-1}=\left(\pscale^{-2}\zeta^{ab}\bnabla_a\pscale\bnabla_b\pscale\right)^{-1}$ is finite and we recall that its boundary value $\lambda_0 = \iota^* \bar \lambda = \iota^* \nu^{-1}$ does not depend on the gauge. Making use of \eqref{GaugeChangeLambdaBar}, one may always achieve $\bar \lambda = \nu^{-1}$. From the proof of Proposition \ref{Prop: BoundaryGaugeStep2} one then sees that it is possible to modify this gauge in order to achieve $f=0$ and that, in addition, one can do so while preserving the condition $\bar \lambda = \nu^{-1}$. We call $L'_{\ei{a}}$ the gauge achieved in this way and $\nabla'$ the unique connection which, by Lemma \ref{Lemma: BoundaryGaugeStep1} and our assumption, now gives $\lambda_A = \nu^{-1} Y_A$. This gives a pair $(\nabla',L'_{\ei{a}})$ satisfying \eqref{Gauge fixing conditions} in a neighbourhood of $x_0$; we also have that, by construction and from \eqref{Extended metric expression, in boudary scale 2 a)}, $\nu = \nu'$.

From this pair, we would now like to construct $L_{\ei{a}}$ such that the pair $(\nabla, L_{\ei{a}})$ also satisfies \eqref{Gauge fixing conditions}. Introduce the gauge/scale transformation  \[L'_\ei{A} \mapsto L'_\ei{A}+\chi_A\Pi^A_{\st\ei{A}}, \quad  \nabla_a' \mapsto \bnabla'_a -\Upsilon_a=\bnabla_a,\] where $\chi_A$ is the only unknown. 

Since $\bar \lambda  =\bar \lambda ' = \bar \lambda -2 \pscale \chi_0$ we must have $\chi_0=0$ and we find, from \eqref{Change of gauge conditions}, that $\chi_a=\chi_A {W'}_a^A$:
\begin{align}\label{Proof: existence of the gauge sufficient conditions}
	-\Upsilon_a &= \nu \pscale\chi_a,&
	0	&= \zeta^{ab}\chi_a \chi_b +2\nu^{-1} N^a \chi_a.
\end{align}
We will be able to solve the first equation if and only if we can show that $\Upsilon_a$ vanishes along $\B$. Furthermore, since $N'^a=N^a + \pscale^{-1}\zeta^{ab}\Upsilon_a$, and both $\nabla'$ and $\nabla$ are regular along $\B$ (the first by Proposition \ref{Proposition: Extended metric expression in gauge1}, the second by hypothesis) we may write: \[ \Upsilon_a= f_0\nabla_a \pscale + f_1\pscale \Upsilon'_a\] with $\Upsilon'_aN^a=0$. However, by assumption, \[\begin{aligned} \nu' &= \nu +2\pscale N^a\Upsilon_a+\zeta^{ab}\Upsilon_a\Upsilon_b = \nu &\Rightarrow& &2\pscale N^a\Upsilon_a +\zeta^{ab}\Upsilon_a\Upsilon_b=0, \end{aligned}\]
and thus if we can prove the first condition of \eqref{Proof: existence of the gauge sufficient conditions} the second will follow immediately. Now, using this equation and the decomposition of $\Upsilon$ we see that:
\[2\pscale \nu f_0 =- \Upsilon_a\Upsilon_b\zeta^{ab}=-f_0^2\pscale^2\nu -f_1^2\pscale^2\zeta^{ab}\Upsilon'_a\Upsilon'_b,\]
which shows that $f_0$ must be of the form $\pscale f_0'$ and hence we can solve for $\chi_a$ in the first equation~\eqref{Proof: existence of the gauge sufficient conditions}.
\end{proof}

Proposition \eqref{Proposition: existence and unicity of the gauge} is the main result of this section and closes it. Let us however record the following.

\begin{prop}\label{Proposition: lambda N is geodesic}
	If $\nabla$ is a scale such that $N^a$ is finite and non-vanishing then
	$|\nu|^{-\frac{1}{2}}N^a\in\E^a(-2)$ is a solution of the projectively invariant geodesic equation, i.e.  
	\[|\nu|^{-\frac{1}{2}}N^c\nabla_c(|\nu|^{-\frac{1}{2}}N^a)=0, \]
	in particular the integral lines of $N^a$ are projective geodesics for $\nabla$.
\end{prop}
\begin{proof}\mbox{}
	
	By the previous proposition, for any scale $\nabla$ such that $N^a$ is finite, we can always find a unique compatible gauge $L_{\ei{a}}$. Recall that, in such choices of gauge, $\bar\lambda N^a=J^AZ_{\st A}^a$,  hence: \[ \bnabla_c (\bar\lambda N^a) =(\bnabla_cJ^A)Z^a_{\st A} -J^AY_A\delta^a_{\st c}=(\bnabla_c J^A)Z^a_{\st A}.\]
	Lemma~\ref{PropertiesExtendedMetricComponents} then shows that:
	\[ \nabla_c(\bar\lambda N^a) = -Z_A^aH^{AB}\upsilon_{Bc}=-(\zeta^{ab}\upsilon_{bc}-\pscale N^aX^B\upsilon_{Bc}).\]
	Observe that since $f=0$ in our gauge, we also have:
	\[ J^B\upsilon_{Bc}=\bar\lambda N^b\upsilon_{bc}=0, \]
	which shows that
	\[\bar\lambda N^c \nabla_c(\bar\lambda N^a) = \pscale \bar\lambda X^B\upsilon_{Bc}N^c N^a.\]
	Now, using equation \eqref{Extended connection expression in boudary scale 2}, we deduce that
	\[\bar\lambda N^c \nabla_c(\bar\lambda N^a)=N^a\frac{1}{2}\bar\lambda N^c\nabla_c\bar\lambda=N^a|\bar\lambda|^{\frac{3}{2}}N^c\nabla_c|\bar\lambda|^\frac{1}{2},  \]
	and thus
	\[ |\bar\lambda|^\frac{1}{2}N^c\nabla_c(|\bar\lambda|^{\frac{1}{2}}N^a)=-|\bar\lambda|^\frac{1}{2}N^c(\nabla_c|\bar\lambda|^\frac{1}{2})N^a+N^c\nabla_c(|\bar\lambda| N^a)=0.\]
\end{proof}

\subsection{Existence for Ricci-flat projectively compact Einstein metrics}\label{ssec:existence}
In the previous discussion, we explored the structures induced by a holonomy reduction of a non-effective projective geometry on a manifold $\Q$. In particular, we studied a gauge fixing condition~\eqref{Gauge fixing conditions}, that, for every projective scale $\nabla$ satisfying the condition that $N^a$ is regular at the boundary, singled out two possible distinguished choices of gauge. At interior points, both solutions were possible and corresponded either to the metric gauge (see Theorem~\ref{Theorem: UniquenessOnInside}) or another gauge $L_{\ei{a}}$ that is uniquely associated with the scale (see Proposition~\ref{Proposition: Extended metric expression in gauge2} and Proposition~\ref{Proposition: existence and unicity of the gauge}). Near points of the boundary $\B$ where $\bar\lambda \neq0$ (which we recall is a gauge invariant condition at boundary points contrary to other points) the metric gauge is always singular and the second solution is the only regular gauge satisfying our conditions. Note that these results have been obtained under the assumptions that the initial non-effective Cartan geometry was well-defined and smooth everywhere and one might be worried that this could constrain the underlying projectively compact Einstein manifold. As we shall now see, this is not the case.\\

The work of the previous section indeed puts us in a position to study the question of the existence, on $M\setminus \mathcal{H}_0$ and for any projectively compact metrics as reviewed in Section~\ref{ClassicalProjectiveCompactification}, of such non-effective Cartan geometries: we only need to check that the expressions given by Theorem~\ref{Theorem: UniquenessOnInside} and Proposition~\ref{Proposition: Extended metric expression in gauge2} lead, when applied to projectively compact metrics, to a well-defined structure.

On the interior $\Mint$, the metric gauge makes things particularly easy because the components of $H^{\ei{A}\ei{B}}$ and the connection, as given by Theorem~\ref{Theorem: UniquenessOnInside}, are clearly well-defined on $\Mint$. We can see these objects as living on a trivial bundle $\T|_{\Mint}\oplus \R$ over $\Mint$.

Near any boundary point $x_0 \in \mathcal{H}_\pm$, the key observation is that any scale associated to a density $\tau$ given by Theorem~\ref{Projective compactification review: CapGoverBoundaryDefiningFunction} satisfies the hypothesis of Proposition~\ref{Proposition: existence and unicity of the gauge} in a neighbourhood $U$. Hence, we can locally define the components using the expressions of Proposition~\ref{Proposition: Extended metric expression in gauge2} (or, equivalently, by starting from the metric gauge and applying the change of gauge \eqref{RequiredChangeFromMetricGauge}). It only remains to check that the corresponding expressions extend smoothly to boundary points:
\begin{theo}\label{Theorem: connection coefficient in constructed gauge}\mbox{}
	
For any scale $\tau$ as given in Theorem~\ref{Projective compactification review: CapGoverBoundaryDefiningFunction}, we define the extended connection modelled on \eqref{Homogeneous space model} by \eqref{ExtendedConnection} and
\begin{equation}
\upsilon_{Ab}=\nu^{-2}N^cP_{cb}Y_A + \pscale^{-1}\left(\nu^{-1}\bP_{ab} -\nu^{-2}\nabla_a\pscale N^dP_{db} -q_{ab}\right)Z^a_{\st A}.
\end{equation}
Because of the identities \eqref{Projective compactification review: identities on schouten}, which are valid for any projectively compact Einstein metric, this expression is finite along the boundary and admits the following asymptotic behaviour:
\begin{align*}
\upsilon_{Ab}=&  \big(2 \nu^{-3} N^c N^ dP_{cd}\big)\nabla_b\pscale\,Y_A \\&+ \pscale^{-1}\left(\nu^{-1}\bP_{ab} -\nu^{-2}\nabla_a\pscale N^dP_{db} -q_{ab}\right)Z^a_{\st A}\\& + O(\pscale).
\end{align*}
This connection admits a holonomy reduction given by a parallel transported non-degenerate tractor pairing $H^{\ei{a}\ei{b}}$:
\begin{align}
	H^{\ei{A}\ei{B}}&=H^{AB}\rotpi^{\ei{a}}_{\,\,A}\rotpi^{\ei{B}}_{\,\,B} + 2\nu^{-1}N^a\, W^A_{a}I^{(\ei{B}}\rotpi^{\ei{A})}{}_{\!\!A}.
\end{align}
It extends smoothly to the boundary and is invertible.
\end{theo}
\begin{proof}From Equation~\eqref{TransformationRuleCotractorValuedConnectionForm}, equation~\eqref{RequiredChangeFromMetricGauge} and Proposition~\ref{Theorem: UniquenessOnInside} we have:
\begin{equation*} \hat{\upsilon}_{Ab}=-\frac{1}{2}\nu^{{-2}}\pscale^{-1}\nabla_b \nu Y_A +\left(-\pscale^{-1}\zeta_{ba}+\pscale^{-2}\nabla_b\pscale\nabla_a\pscale +\nu^{-1}\pscale^{-1}P_{ba}+\frac{1}{2}\nu^{-2}\pscale^{-2}\nabla_b\nu\nabla_a\pscale\right)\;Z^a_A.  \end{equation*}
This also follows from Proposition \ref{Proposition: Extended metric expression in gauge2}.
The desired expression in the statement is then obtained from Equation~\eqref{Projective compactification review: asymptotic form of the metric} and the identities \eqref{Projective compactification review: identities on schouten}.
The identities \eqref{Projective compactification review: identities on schouten} also justify that it extends to boundary points.

The expression for $H^{\ei{A}\ei{B}}$ follows from Proposition \ref{Proposition: Extended metric expression in gauge2}. Since $N^a\nu^{-1}$ is non-vanishing, $H^{\ei{a}\ei{b}}$ is non-degenerate.
\end{proof}

At this point there are still two potential difficulties that we will elucidate in one go. The first is that Theorem~\ref{Theorem: connection coefficient in constructed gauge} makes use of a scale $\tau$ and since we could have chosen instead any other scale $\tilde\tau$, we need to check that all such connections only differ by a (finite) gauge transformation. This will ensure that the resulting connection is a well defined (i.e. invariant) object.
Second, we have only shown that the structure exists locally, and need to justify that the local constructions can be glued together in a coherent fashion. 

Since there is a canonical way to move between any of these gauges $L_\ei{A}$, $\tilde L_\ei{A}$ and the metric gauge $\mL_\ei{A}$ (which is well-defined in $\Mint$), we can define a canonical way to move between two local gauges determined by two densities $\tau$ and $\tilde\tau$ on the intersection of two local neighbourhoods $U\cap V$ by the successive gauge/scale transformations:
\[(L_\ei{A}, \nabla_a) \to (\mL_\ei{A}, \nabla_a) \to (\mL_\ei{A}, \tilde \nabla_a) \to (\tilde L_\ei{A},\tilde \nabla). \]
If the overall gauge transformation cotractor $\chi_A$, i.e. defined such that $\tilde{L}_{\ei{B}}=L_{\ei{B}}+\chi_A\,\Pi^A_{\,\,\ei{A}}$, extends smoothly to boundary points, then this gives a valid gauge transformation between the local structures that can be used to extend the structure to $U\cup V$. Along the way it will also prove that the extended connection as defined in Theorem~\ref{Theorem: connection coefficient in constructed gauge} is an invariant object.

\begin{lemm}\label{boundary gauge change of tau} Let $\tau$, $\tilde \tau$ be two scales as given by Theorem~\ref{Projective compactification review: CapGoverBoundaryDefiningFunction}, and defined in a neighbourhood of $\mathcal{H}_\pm$. Denote by $L_\ei{a}$, $\tilde L_\ei{a}$ the corresponding gauges determined by Proposition~\ref{Proposition: existence and unicity of the gauge}. Let us write
$\tilde\tau = \tau \mp \omega \pscale$, then in the splitting determined by $\tau$:
\[\tilde L_\ei{a}= L_\ei{a} \;+ D_A(\, \omega \tau\,)\,\Pi^A_{\st\ei{A}} +  O(\pscale, \bnabla \pscale)\]
From which we deduce:
\begin{equation}\label{change in upsilon along boundary 1}
	 \tilde\upsilon_{ab}= \upsilon_{ab}\; +\,\big(\, \tau\,\nabla_a\nabla_b \omega\,\pm \tau^{-1}\,\bar{h}_{ab}\, \omega \big) + O(\pscale, \bnabla\pscale).
\end{equation}
\end{lemm}
\begin{proof}
The proof of the first equation can be found in Appendix~\ref{sec:ProofChangeOfupsilon}. It is essentially a computation that provides no interesting insight at this point. One can also check the result by solving the gauge fixing conditions \eqref{Change of gauge conditions} for $(\chi_A, \Upsilon_a = -\tilde\tau^{-1}\nabla_a \tilde\tau)$ at leading order.  

To obtain the second equation it is sufficient to take the derivative of the previous equation.
Since $\nabla_c I_B=0$, it follows that $\nabla_a\nabla_b \pscale = -P_{ab}\pscale$, hence:
\[ \nabla O(\pscale, \nabla \pscale)= O(\pscale, \nabla \pscale). \]
Now:
\[ \nabla_a D_B(\omega\nu^{-1}\tau^{-1})= \nu^{-1}\tau^{-1}(P_{ab}\omega +\nabla_a\nabla_b \omega)Z^b_B +O(\pscale), \]
and Equation~\eqref{change in upsilon along boundary 1} then follows directly from using \eqref{Projective compactification review: identities on schouten}  and \eqref{Projective compactification review: identities on lambda nu}.
\end{proof}

The general spirit of this construction is simply that we can deduce from this a Cartan atlas $\{U^i, \omega^i\}$ and transition functions $\{f_{ij}\}$ that can be used to reconstruct the $P$-principal bundle equipped with its global Cartan connection. We then recover the extended tractor bundle by the standard associated bundle construction. This construction is at this point ad-hoc and a more satisfying geometric description of the extended tractor bundle $\mathscr{T}$ is the object of work in progress.
This proves the following theorem.
\begin{theo}\label{Thm:ExtensionToBoundary}\mbox{}
	
	 Given a projectively compact Ricci flat Einstein metric $\pg_{ab}$ of signature $(p,q)$ on the interior $\Mint$ of a manifold with boundary $\M$, there is a uniquely determined non-effective Cartan connection modelled on \eqref{Homogeneous space model} and extending to $\M\setminus\mathcal{H}_0$. 

This Cartan geometry is such that, by construction, each choice of scale $\tau \in \mathcal{E}(1)$ as given by Theorem~\ref{Projective compactification review: CapGoverBoundaryDefiningFunction}, uniquely determines a gauge close to any point $x_0 \in \B\setminus\mathcal{H}_0$.

What is more, it admits a holonomy reduction to the subgroup $ISO(p,q)$ of $PSL(n+2)$:
\[ \begin{pmatrix}1 & -\omega^\rho\eta_{\rho \sigma}A^\sigma_{\st \nu} & {-\frac{1}{2}\eta_{\mu\nu}\omega^\mu\omega^\nu} \\ 0 & A^\mu_\nu & \omega^\mu \\ 0 & 0 & 1 \end{pmatrix} \mod Z(SL(n+2)), \quad A \in SO(p,q). \]
\end{theo}

\section{Projective structure on the extended boundary}\label{sec:Projective structure on the extended boundary}
\subsection{Induced non effective Cartan geometry on the boundary}\label{ssec:DependenceTau}

Let us recall, from \cite{Cap:2014aa,Cap:2014ab}, that the projective boundary $\mathcal{H}_{\pm}$ inherits a projective Cartan geometry together with a holonomy reduction to $SO(1,n-1)$. This induces on $\mathcal{H}_{\pm}$ a metric of signature\footnote{$(r,s)= (1,n-2)$ on $\mathcal{H}_{+}$ and $(r,s)= (0,n-1)$ on $\mathcal{H}_{-}$.} $(r,s)$, and a Cartan geometry modelled on 
	\begin{equation}
	SO(1,n-1) \;/\; SO(r,s).
\end{equation}

Now, in the previous sections, we found that every choice of scale $\tau$ given by Theorem~\ref{Projective compactification review: CapGoverBoundaryDefiningFunction}, i.e. such that\footnote{In this section we will also frequently use the notation $\tau_0 := \iota^* \tau = |\lambda_0|^{\frac{1}{2}}$ for this density.} $\iota^* \tau = |\lambda_0|^{\frac{1}{2}}$ gives rise to a distinguished boundary gauge $L_{\ei{a}}$  for the non-effective projective geometry.  We obtain in this way, along the boundary, and after making use of the holonomy reduction to $ISO(1,n-1)$, a non-effective Cartan geometry modelled on 
	\begin{equation}
	SO(1,n-1)\ltimes \mathbb{R}^{n}  \;/\; SO(r,s) \ltimes \mathbb{R}^{n},
\end{equation}
and related to the effective one by ``forgetting'' about the non-effective part $\mathbb{R}^{n}$; in practice this will mean quotienting by $I^{\ei{a}}$ in the expressions below.

The induced connection can be read off from Eqs. \eqref{ExtendedConnection}, \eqref{TractorConnection}, \eqref{Projective compactification review: identities on schouten} and Theorem \ref{Theorem: connection coefficient in constructed gauge}: fixing a section $\tau$, and denoting by $\bar W^A_a$ the splitting of the projective tractor bundle on $\mathcal{H}_{\pm}$ determined by $\tau_0 = |\lambda_0|^{\frac{1}{2}}$, it is given by
\begin{align}\label{Induced non effective Cartan connection}
	\begin{aligned}
	&\nabla_a \big(\tau_0^{-1} X^{A}\rotpi^{\ei{a}}_{\,\,A}\big)=  \tau_0^{-1}\bar W^{A}_a \rotpi^{\ei{a}}_{\,\,A},\\[0.3em]
	&\nabla_a \big(\tau_0^{-1} \bar W^{A}_b\rotpi^{\ei{a}}_{\,\,A}\big)=\mp \tau_0^{-2}\,\bar h_{ab}\;\left(\tau_0^{-1}X^{A}\rotpi^{\ei{a}}_{\,\,A}\right) + \tau_0^{-1}\iota^*\upsilon_{ab} \,I^{\ei{a}},\\[0.3em]
	&\nabla_a I^{\ei{a}}=0.
	\end{aligned}
\end{align}
with $\iota^*\upsilon_{ab} = \iota^*\big( \pscale^{-1} \left( \nu^{-1}\bP_{ab} -q_{ab}\right)\big)$.

 There is some interplay between Lemma~\ref{boundary gauge change of tau} and this induced Cartan connection on which we will now elaborate.

Indeed, this lemma shows that, although all the projective scales $\tau$ coincide in restriction to the boundary, the restriction of the corresponding gauges $L_{\ei{a}}$ actually depends on the details of their extensions in the interior. Nevertheless, the dependence is relatively tame: it only depends on the section of $\eB$ determined by $\tau$.  

\begin{theo}\label{Change of upsilon}\mbox{}
	
Let $\tau$, $\tilde{\tau}$ be two distinguished projective scales as given by Theorem~\ref{Projective compactification review: CapGoverBoundaryDefiningFunction} and let us write 
\[ \tilde\tau= \tau\; \mp \omega\pscale.\] 
for some function $\omega$ in a neighbourhood of $\mathcal{H}_{\pm}$. Then

\begin{itemize}
\item The connection component $\upsilon_{ab}$ of the extended tractor connection has, when \textbf{pulled back to the boundary} along the canonical injection $\iota : \B \hookrightarrow M$, the following transformation rule: 
\begin{equation}\label{change in upsilon along boundary} \iota^*\tilde{\upsilon}_{ab}=\iota^*\upsilon_{ab}\; +\iota^*\Big(\,\tau_0\, \nabla_a\nabla_b \omega \pm  \tau^{-1}_0\, \bar{h}_{ab}\omega  \,\Big).\end{equation}
\item In particular, if $\tau$ and $\tilde{\tau}$ define the same section
\[ \hat{\tau}: \mathcal{H}_{\pm} \rightarrow \eB \]
of $\eB \to \mathcal{H}_{\pm}$, i.e. $\omega=\sigma \omega_1$, $\omega_1\in \E(-1)$, then the $\iota^*\tilde\upsilon_{ab}=\iota^*\upsilon_{ab}$.
\end{itemize}
 Hence, each section of $\eB$ globally defines $\upsilon_{ab} \in \E^\B_{(ab)}(1)$ and the boundary value of the extended tractor connection only depends on this section.
\end{theo}
\begin{proof}
The first point is just obtained by pulling back the result of Lemma~\ref{boundary gauge change of tau} along the canonical injection.
To prove the second point, observe that if $\omega=\pscale \omega_1$ then:
\[ \nabla_a\nabla_b \omega = \pscale \nabla_a\nabla_b \omega_1 + 2\nabla_{(a} \pscale \nabla_{b)}\omega_1 + \underbrace{(\nabla_a\nabla_b\pscale)}_{-P_{ab}\pscale} \omega_1.\]
\end{proof}

\subsection{Carrollian structure on the extended boundary}

The simple transformation rule \eqref{change in upsilon along boundary} for the component $\upsilon_{ab}$ of the induced tractor connection \eqref{Induced non effective Cartan connection} suggests that it would perhaps more naturally live as a structure on the extended boundary $\eB$. In this section, we shall show that this intuition is correct and use the properties of the non-effective connection on $\mathcal{H}_{\pm}$ to construct an \emph{effective} Cartan geometry on $\eB$.

First, let us make some general remarks related to the fact that $\eB \to \mathcal{H}_{\pm}$ is a principal $\R$-bundle and that it is equipped with a preferred degenerate metric $h_{ab}:=\pi^*(\tau_0^{-2} \bar{h}_{ab})$ of signature $(1,r,s)$, where $(r,s)= (1,n-2)$ on $\mathcal{H}_{+}$ and $(r,s)= (0,n-1)$ on $\mathcal{H}_{-}$. In a local trivialisation we shall write points of $\eB$ as couples $p=(x,u)$ with $n^a = \partial_u$ the generator of the $\R$-action. The pair $(h_{ab}, n^a)$ satisfies
\begin{align*}
	n^b h_{ab} &=0, &\mathcal{L}_n h_{ab} &=0,
\end{align*}
 and is a (weak) Carrollian geometry in the sense of \cite{duval_carroll_2014}. It naturally defines a reduction of the frame bundle:
\begin{defi}
	The pair $(h_{ab}, n^a)$ on $\eB$ defines a reduction of the structure group of the frame bundle from $GL(n)$ to $SO(r,s) \ltimes \R^{n-1}$:
	
	Let $\mathscr{F}$ denote the bundle of $1$-jets at $(0,0)$ of local diffeomorphisms $\Phi : \mathbb{R}^{n-1}\times\R \rightarrow \eB$ such that there is a local diffeomorphism $\phi : \R^{n-1}\rightarrow \B$ making the following diagram commute:
	\begin{center}
		\begin{tikzcd}\mathbb{R}^{n-1}\times\R \arrow[r,"\Phi"]  \arrow[d,"\textnormal{Proj}_1"]& \eB \arrow[d,"\pi"] \\ \R^{n-1} \arrow[r,"\phi"]&\B \end{tikzcd}
	\end{center}
	and: \begin{enumerate}
		\item $\Phi_*\frac{\partial}{\partial x^n} =\frac{\partial}{\partial u}$,
		\item $\Phi^*h= \eta_{ij}\dd x^i\dd x^j$, $\eta=\textnormal{diag}(\mp, 1, \dots ,1)$
	\end{enumerate}
	This is a principal bundle with structure group $SO(r,s) \ltimes \mathbb{R}^{n-1}$ formed by matrices of the form:
	\[\begin{pmatrix}1 & \chi \\ 0 & A \end{pmatrix}, \quad A \in SO(r,s).\]
	The right action of $(\R,+)$ on $\eB$ lifts to an action on $\mathscr{F}$ which commutes with the action of $SO(r,s) \ltimes \R^{n-1}$.
\end{defi}
It is important to observe at this point that this bundle admits specific types of (local) sections $t$, coming from sections $\hat{\tau}$ of $\eB$, that have the equivariance property: \[ t\circ R_\lambda= R_\lambda \circ t, \quad \lambda \in \R.\] Throughout, we shall only work with sections of this type; that we will refer to as \emph{admissible} frames. Any two such admissible frames are related by a map $\mathcal{H}_{\pm} \rightarrow SO(r,s)\ltimes\R^{n-1}$ (so they are transported by the right action along the fibre). By construction these frames must be of the form $(n^a, m_i{}^a)$ and the corresponding dual co-frames $(\nabla_a u , m^i{}_a)$.

 Since, under our assumption, $\nabla_n u=1$, the zero set of $u$ always define a section $\hat{\tau} : \mathcal{H}_{\pm} \to \eB$. The other way round a section $\hat{\tau}$ of $\eB \to \mathcal{H}_{\pm}$ uniquely defines a function $u$ as the unique function such that $\nabla_n u=1$ and $u\circ \hat{\tau}=0$. Therefore, one sees that a section $\hat{\tau}$ defines, up to $SO(r,s)$, an admissible frame $(\nabla_a u , m^i{}_a)$.
 
  This remark will serve to state the following.
 
\begin{theo}\mbox{}\label{Theorem: Cartan geometry on Ti/Spi}
	
	The non-effective Cartan geometry \eqref{Induced non effective Cartan connection} on $\mathcal{H}_{\pm}$ induces on the extended boundary $\eB$ a Cartan geometry modelled on 
	\begin{equation}
		SO(1,n-1)\ltimes \mathbb{R}^{n}  \;/\; SO(r,s) \ltimes \mathbb{R}^{n-1}.
	\end{equation}
	These homogenous spaces are the (pseudo)-Carrollian spaces $\Ti$ and $\Spi$ from \cite{Figueroa-OFarrill:2021sxz} and in particular yield  \underline{effective} Cartan geometries.
	
	In an admissible frame, given by a section $\hat{\tau}$ of $\eB \to \mathcal{H}_{\pm}$, the connection is given by
\begin{align}\label{Induced effective Cartan geometry}
	\begin{aligned}
	&\nabla_a \Big(\tau_0^{-1} X^{A}\rotpi^{\ei{a}}_{\,\,A}\Big)=  \tau_0^{-1}\, \pi^* \bar W^{A}_a\rotpi^{\ei{a}}_{\,\,A} + \nabla_a u\, I^{\ei{a}},\\[0.3em]
	 &\nabla_a \Big( \tau_0^{-1} \, \pi^*\bar W^{A}_b\rotpi^{\ei{a}}_{\,\,A}\Big)=\mp h_{ab} \,\left(\tau_0^{-1} X^{A} \rotpi^{\ei{a}}_{\,\,A}\right) + \Big(\tau_0^{-1} \upsilon_{ab}\; \pm u \, h_{ab}\Big) I^{\ei{a}},\\[0.3em]
	 &\nabla_a I^{\ei{a}}=0.
	 \end{aligned}
\end{align}
\end{theo}
  By results from \cite[Section 2]{Herfray:2021qmp} these Cartan geometries are equivalent to strong Carrollian geometries in the sense of \cite{duval_carroll_2014} (see also appendix of \cite{bekaert_connections_2018} for more details on these geometry).
\begin{proof}
Let $\tau$ and $\tilde\tau = \tau \pm \omega \pscale$, we only need to check that the transformation rules of the above objects are consistent with those of a Cartan connection. In fact, only the term containing $\upsilon$ might pose a problem.

Since by definition $u\circ \tau =0$ and $u\circ \tilde\tau =\omega$, one has $u\mapsto u -\omega$ and, from \eqref{change in upsilon along boundary},
\begin{align*} \tau_0^{-1}\upsilon_{ab} \pm u\, \, h_{ab} \mapsto & \quad \tau_0^{-1}\upsilon_{ab}+ \nabla_a\nabla_b\omega \pm \, h_{ab}\omega \pm (u-\omega) \, h_{ab}\\
	&=\tau_0^{-1} \upsilon_{ab}\pm u \, h_{ab} +  \nabla_a\nabla_b\omega
\end{align*}
which is the correct transformation rule, see e.g. \cite[Section 2]{Herfray:2021qmp}. We only need to check that this geometry is torsion free, this can be done by directly evaluating the curvature and making use of $\upsilon_{[ab]}=0$ (see the explicit expression below $\eqref{Induced non effective Cartan connection}$).
\end{proof}

The precise relation between this effective Cartan geometry on $\eB$ and the non-effective Cartan geometry realised on $\mathcal{H}_{\pm}$ is given by the following.

\begin{prop}\label{Proposition: non effective Cartan geometry from effective}
	The curvature $F^{\ei{a}}{}_{\ei{b}}{}_{cd}$ of the Cartan geometry given by Theorem \ref{Theorem: Cartan geometry on Ti/Spi} satisfies $n^c F^{\ei{a}}{}_{\ei{b}}{}_{cd}=0$. In particular, one can quotient by the action of $(\R,+)$  and obtain in this way a (non-effective) Cartan geometry on $\mathcal{H}_{\pm}$. The resulting connection then coincides with \eqref{Induced non effective Cartan connection}.
\end{prop}
\begin{proof}
The constraint on the curvature can be checked by a direct computation. Then a tractor field 
\begin{equation*}
\tilde T^{\ei{a}} = \tilde T^0 I^{\ei{a}} + \tilde T^b \,\Big(\tau_0^{-1} \pi^*\bar W^{A}_b\rotpi^{\ei{a}}_{\,\,A}\Big) + \tilde T^-\; \Big(\tau_0^{-1} X^{A}\rotpi^{\ei{a}}_{\,\,A}\Big)
\end{equation*}
is parallel transported along the fibres of $\eB \to \mathcal{H}_{\pm}$ if $n^c\nabla_cT^{\ei{a}}=0$ i.e. iff
\begin{align*}
	\partial_u \tilde T^0 +  \tilde T^- &=0,&
	\partial_u \tilde T^a  &=0,& \partial_u \tilde T^-  &=0&.
\end{align*}
In other words, vertically constant tractors are related to tractors on $\mathcal{H}_{\pm}$
\begin{equation*}
	T^{\ei{a}} = T^0 I^{\ei{a}} + T^b \,\Big(\tau_0^{-1}\bar W^{A}_b\rotpi^{\ei{a}}_{\,\,A}\Big) + T^-\; \Big(\tau_0^{-1} X^{A}\rotpi^{\ei{a}}_{\,\,A}\Big)
\end{equation*}
via  $(\tilde T^0, \tilde T^a, \tilde T^-) = ( T^0 - u T^-, T^a, T^-)$. One can now check that differentiating $\tilde T^{\ei{a}}$ in the horizontal direction defines on $\mathcal{H}_{\pm}$ a connection identical with \eqref{Induced non effective Cartan connection}.

\end{proof}

\subsection{Projective structure on the extended boundary}

From Theorem \ref{Theorem: Cartan geometry on Ti/Spi}, and results of \cite[Section 2]{Herfray:2021qmp}, we obtained that $\eB$ is equipped with a torsion free linear connection $\nabla$
satisfying
\begin{align}
	\nabla_c h_{ab} &=0, & \nabla_c n^a &=0.
\end{align}
(This is the strong Carrollian from \cite{duval_carroll_2014}). In particular $\eB$ is equipped with a projective connection $[\nabla]$. 

We now want to investigate aspects of this projective geometry. It will prove useful to introduce the following densified objects\footnote{In order not to proliferate notation, we abusively identify the densified objects on $\mathcal{H}_{\pm}$ and their pullback on $\eB$.},
\begin{align*}
	\bar{n}^a &:= \tau_0^{-1} n^a, & \bar h_{ab} = \tau_0^{2} h_{ab}.
\end{align*}
Recall that then $\lvert\detd(\bar h^{ab})\lvert = \tau_0^2$ and therefore the pair $(\bar{n}^a,\bar h_{ab})$ really is equivalent to $(n^a,  h_{ab})$. 

Using $\n^a$ one may reproduce the steps of the proof of~\cite[Proposition 5.1]{RSTA20230042} and obtain the isomorphism:
\begin{lemm} \label{PullBackOfDensitesOnBase}
	\[ \E(1)_{\eB} \simeq \pi^*\E_{\B}(1). \]
\end{lemm}
One may also show that:
\begin{lemm} There is a well-defined Lie derivative operator $\mathcal{L}_n : \mathcal{E}_{\eB}(1) \rightarrow C^\infty(\R).$ 
	\begin{enumerate}\item Projective densities that are pullbacks of densities on $\B$ will be called \textbf{adapted scales}, they can be characterised by: \[ \tau =\pi^*\bar\tau \Leftrightarrow \mathcal{L}_\n\tau=0.\] 
	\item Since dual tractors $\T^*_\eB$ are identified with $J^1\E(1)_\eB$, the Lie derivative $\mathcal{L}_\n$ defines, via the bundle map $j^1_p\sigma \mapsto (\mathcal{L}_\n\sigma)_p$, a canonical tractor field $I^\pti{A}$. Furthermore, we have the short exact sequence of vector bundles:
	\begin{center}
\begin{tikzcd}
 0 \arrow[r] & \pi^*\T^*_{\B} \arrow[r,"\Pi"] & \T^*_{\eB} \arrow[r,"I"] &\R \arrow[r] & 0.
\end{tikzcd}
\end{center}
	\end{enumerate}
\end{lemm}

\begin{prop}
	The pair $(\bar{n}^a, \bar h_{ab})$ defines on $\eB$ two projective tractor field $I^{\pti{a}}$ and $H_{\pti{a}\pti{b}}$ given, in an adapted scale, by
	\begin{align}
		I^{\pti{a}}& = \bar n^a W_a{}^{\pti{a}}, & H_{\pti{a}\pti{b}} & = (\bar h_{ab} - \nabla_a \tau_0\nabla_b \tau_0) Z^a_{\pti{a}}Z^b_{\pti{b}} -  \tau_0 \nabla_a \tau_0 \;2Z^a{}_{(\pti{a}} Y_{\pti{b})} + \tau_0^2 Y_{\pti{a}}  Y_{\pti{b}}.
	\end{align}
	The inner product $H_{\pti{a}\pti{b}}$ is degenerate and $I^{\pti{a}}H_{\pti{a}\pti{b}}=0$.
\end{prop}
\begin{proof}
	One only needs to check that, upon changing adapted scale, these tractors are invariant. However this is the case since then
	\begin{align*}
		\nabla_a \tau_0 &\mapsto \nabla_a \tau_0 + \Upsilon_a \tau_0,& W_a{}^{\pti{a}} &\mapsto W_a{}^{\pti{a}} + \Upsilon_a X^{\pti{a}},\\
		Z^a_{\pti{a}} & \mapsto Z^a_{\pti{a}}, & Y_{\pti{a}} & \mapsto Y_{\pti{a}} - \Upsilon_a Z^a_{\pti{a}},
	\end{align*}
	and $n^a\Upsilon_a =0$ due to the fact that we are restricting to adapted scales.
\end{proof}

As was already suggested by Proposition \ref{Proposition: non effective Cartan geometry from effective}, the Cartan geometry from Theorem \ref{Theorem: Cartan geometry on Ti/Spi} is not generic, rather one has.

\begin{prop}\label{Prop:ProjectiveStructureHolonomyReduction}\mbox{}
	
	The Cartan geometry given by Theorem \ref{Theorem: Cartan geometry on Ti/Spi} is obtained from the projective Cartan geometry on $\eB$ by a holonomy reduction to $ISO(1,n-1)$. 
	
	This follows from the fact that $I^{\pti{a}}$ and $H_{\pti{a}\pti{b}}$ are parallel transported
	\begin{align*}
		\nabla_c I^{\pti{a}} &=0, & \nabla_c H_{\pti{a}\pti{b}} &=0. 
	\end{align*}
\end{prop}
\begin{proof}
	Since the content of the above proposition is written in a projectively invariant form, we can prove it in any scale that we like. In the adapted scale $\tau_0$, the projective Cartan connection is simply given by \eqref{Induced effective Cartan geometry} and one can check that it satisfies $\nabla_c I^{\pti{a}} =0$, $\nabla_c H_{\pti{a}\pti{b}} =0$.
\end{proof}

To conclude this section, we observe that it would have also been possible to first construct the projective structure and deduce the Carrollian structures from the holonomy reduction. This point of view would emphasise the strong inspiration of the present work from the picture developed in~\cite{RSTA20230042}. The projective structure obtained here on the extended boundary, and the related question of its projective compactness, would be a natural starting point for recovering conformal null infinity $\scri$ from $\Spi/\Ti$ as the projective boundary; this will be the object of future investigations.

\newpage
\appendix
\section{Proof of Lemma~\ref{boundary gauge change of tau}}\label{sec:ProofChangeOfupsilon}
In this section we resume the notations introduced in Section~\ref{sec:HolonomyReductionOfNonEffective}. We recall that, in the statement of Lemma~\ref{boundary gauge change of tau}, we set $\tilde{\tau}=\tau \mp \omega\pscale$ for some arbitrary function $\omega$, to simplify the presentation we will only treat the case with \enquote{+}. 
Let $\bnabla_a$ (resp. $\tilde\bnabla_a$) be the projective scale determined by $\tau$ (resp. $\tilde\tau$), for convenience, introduce:\[\Upsilon_a= \tilde\tau^{-1}\nabla_a\tilde\tau \Rightarrow \tilde\bnabla_a =\bnabla_a-\Upsilon_a,\]
and:
\[ \eta^a=\bar\lambda N^a = \nu^{-1}\pscale^{-2} \zeta^{ab}\nabla_b\pscale \Rightarrow \eta^a\nabla_a\pscale =1.\]

Our goal is to determine the relationship between the distinguished gauges $L_\ei{A}$ and $\tilde{L}_\ei{A}$ these scales determine by Proposition~\ref{Proposition: existence and unicity of the gauge}. To this end it is sufficient to compute the difference \[\bar{\chi}_A=\tilde{\chi}_A-\tilde{\chi}_A\]
where $\chi_A$ (resp. $\tilde{\chi}_A$) are defined by Eq.~\eqref{RequiredChangeFromMetricGauge}.

Now:
\[ \bnabla_a\tilde\tau = \omega\bnabla_a \pscale + \bnabla_a \omega \pscale.\]
Hence:
\begin{align*}\Upsilon_a=\tau^{-1}\left(\omega\bnabla_a\pscale + (\bnabla_a\omega-\omega^2\tau^{-1}\bnabla_a \pscale)\pscale) \right.&+(\omega^3\bnabla_a\pscale\tau^{-2}-\omega\nabla_a \omega\tau^{-1})\pscale^2 +O(\pscale^3). \end{align*}
We will also need to compute, $\zeta^{ab}\Upsilon_a$. Using that $\zeta^{ab}\nabla_a\pscale =\nu\pscale^2 \eta^a$: 
\begin{align*}
\zeta^{ab}\Upsilon_a&=\tau^{-1}\left(\pscale^2\nu \omega \eta^b +\zeta^{ab}\nabla_a \omega\pscale -\omega\zeta^{ab}\nabla_a \omega \tau^{-1}\pscale^2 + O(\pscale^3)\right)
\\&=\tau^{-1}\pscale\bigg(\zeta^{ab}\nabla_a \omega +(\nu \omega \eta^b-\omega\zeta^{ab}\nabla_a \omega\tau^{-1})\pscale+O(\pscale^2) \bigg)
\end{align*}
By definition, we have:
\[\tilde\nu = \nu -2\pscale \nu \eta^b\Upsilon_b + \zeta^{ab}\Upsilon_a\Upsilon_b.\]
We evaluate the two terms separately, introducing the \emph{weighted} quantities:\[ \nabla_\eta =\eta^a\nabla_a  \quad \textrm{ and } \quad |\nabla \omega|^2=\zeta^{ab}\nabla_a \omega\nabla_b \omega,\] 
one may write:
\[\eta^a\Upsilon_a=\tau^{-1}(\omega +(\nabla_\eta \omega -\omega^2\tau^{-1})\pscale+(\omega^3\tau^{-2}-\omega\nabla_\eta\tau^{-1})\pscale^2+O(\pscale^3))\]
and
\[\zeta^{ab}\Upsilon_a\Upsilon_b=\tau^{-2}\pscale^2(\omega^2\nu + |\nabla \omega|^2) +O(\pscale^3) \]
Hence:
\begin{align*} \tilde\nu = \nu - 2\nu\pscale\tau^{-1}\omega &+ (3\nu\tau^{-2}\omega^2+\tau^{-2}|\nabla \omega|^2 - 2\nu \nabla_\eta\tau^{-1})\pscale^2 +O(\pscale^3)\end{align*}
from which it follows that
\begin{align*}
\tilde\nu^{-1}=\nu^{-1}\big(1+2\pscale\tau^{-1}\omega&+(\omega^2\tau^{-2}-\nu^{-1}\tau^{-2}|\nabla \omega|^2 + 2\nabla_\eta\tau^{-1})\pscale^2+O(\pscale^3)\big).
\end{align*}
It is now straightforward to compute $\tilde\chi_A$ in terms of $\chi_A$. Working in the splitting determined by $\bnabla$, we have:
\[\tilde\chi_A=-\frac{1}{2}\pscale^{-1}\tilde\nu^{-1}Y_A +\frac{1}{2}\pscale^{-2}\tilde\nu^{-1}(\nabla_a\pscale - 2\Upsilon_a\pscale)Z^a_A.\]
The first component is easy seen to be:
\[-\frac{1}{2}\pscale^{-1}\nu^{-1} - \tau^{-1}\nu^{-1}\omega + O(\pscale).\]
For the second, we work order by order, ignoring the factor of $\frac{1}{2}$ and writing:
\[ \nabla_a\pscale -2\Upsilon_a\pscale= \nabla_a\pscale - 2\pscale \omega \nabla_a\pscale\tau^{-1} - 2\pscale^2\tau^{-1}(\nabla_a \omega -\omega^2\tau^{-1}\nabla_a\pscale)+O(\pscale^3)\]
we obtain the following expressions at each order:
\begin{center}
\def\arraystretch{1.5}
\begin{tabular}{c|l}
Order & \\\hline
$\pscale^{-2}$ &  $\nu^{-1}\pscale^{-2}\nabla_a\pscale$\\\hline
$\pscale^{-1}$ &$2\tau^{-1}\omega\nu^{-1}\nabla_a\pscale - 2\nu^{-1} \omega \nabla_a\pscale\tau^{-1}=0$\\\hline
 \multirowcell{3}{$\pscale^{0\phantom{-}}$}&$\nu^{-1}(\omega^2\tau^{-2}-\nu^{-1}\tau^{-2}|\nabla \omega|^2+2\nabla_\eta \omega\tau^{-1})\nabla_a\pscale$ \\& $-4\tau^{-1}\omega^2\nabla_a\pscale \nu^{-1}$ \\&
 $-2\nu^{-1}\tau^{-1}(\nabla_a \omega -\omega^2\tau^{-1}\nabla_a\pscale)$
\end{tabular}
\end{center}
which yields the desired result, after observing that $\nabla \nu =O(\pscale)$ and $\nabla \tau= O(\pscale)$.

\printbibliography[title=Bibliography]
 \end{document}